\documentclass[10pt]{amsart}

\title[Finite length spectra of random surfaces and their dependence on genus]{Finite length spectra of random surfaces\\ and their dependence on genus}
\author{Bram Petri}
\address{Max Planck Institute for Mathematics, Bonn, Germany}
\email{brampetri@mpim-bonn.mpg.de}
\date{\today}
\thanks{Research supported by Swiss National Science Foundation grant number PP00P2\textunderscore 128557}
\keywords{Random Riemann surfaces, hyperbolic surfaces, Bely\v{\i} surfaces, Riemannian surfaces, random graphs, geodesics}
\subjclass[2010]{Primary: 57M50. Secondary: 53C22, 05C80}

\newtheorem{thm}{Theorem}[section]
\newtheorem{prp}[thm]{Proposition}
\newtheorem{cor}[thm]{Corollary}
\newtheorem{lem}[thm]{Lemma}

\newtheorem*{thmA}{Theorem A}
\newtheorem*{thmB}{Theorem B}
\newtheorem*{cor1}{Corollary 1}
\newtheorem*{cor2}{Corollary 2}
\newtheorem*{cor3}{Corollary 3}

\newtheorem{innerthmc}{Theorem}
\newenvironment{thmc}[1]
  {\renewcommand\theinnerthmc{#1}\innerthmc}
  {\endinnerthmc}

\theoremstyle{definition}
\newtheorem{dff}{Definition}[section]

\newtheorem*{obs}{Observation}

\usepackage{amsmath,amsfonts,amssymb,float,graphicx,enumitem}
\usepackage[enableskew]{youngtab}
\usepackage[margin=1in]{geometry}

\newcommand{\Pro}[2]{\mathbb{P}_{#1}\left[#2\right]}
\newcommand{\ProC}[3]{\mathbb{P}_{#1}\left[#2\mid \;#3\right]}
\newcommand{\ExV}[2]{\mathbb{E}_{#1}\left[#2\right]}
\newcommand{\verz}[2]{\left\{ #1 ; #2\right\}}
\newcommand{\rij}[3]{\left\{#1\right\}_{#2}^{#3}}
\newcommand{\norm}[1]{\left|\left| #1 \right|\right|}
\newcommand{\abs}[1]{\left| #1 \right|}
\newcommand{\aant}[1]{\left| #1 \right|}
\newcommand{\tr}[1]{\mathrm{tr}\left(#1\right)}
\newcommand{\floor}[1]{\left\lfloor #1 \right\rfloor}
\newcommand{\ceil}[1]{\left\lceil #1 \right\rceil }
\newcommand{\sys}{\mathrm{sys}}

\newcommand{\symm}[1]{\mathrm{S}_{#1}}
\newcommand{\alt}[1]{\mathrm{A}_{#1}}

\begin{document}

\begin{abstract}
The main goal of this article is to understand how the length spectrum of a random surface depends on its genus. Here a random surface means a surface obtained by randomly gluing together an even number of triangles carrying a fixed metric.

Given suitable restrictions on the genus of the surface, we consider the number of appearances of fixed finite sets of combinatorial types of curves. Of any such set we determine the asymptotics of the probability distribution. It turns out that these distributions are independent of the genus in an appropriate sense. 

As an application of our results we study the probability distribution of the systole of random surfaces in a hyperbolic and a more general Riemannian setting. In the hyperbolic setting we are able to determine the limit of the probability distribution for the number of triangles tending to infinity and in the Riemannian setting we derive bounds.
\end{abstract}

\maketitle

\section{Introduction}
To understand the topological and geometric properties of manifolds, one approach is to look for properties that all manifolds have. One can also look for manifolds that are extremal for a certain property in an appropriate sense. Another option is to look at the properties that a generic or an average manifold has. This last kind of question naturally leads to the study of random manifolds. That is, one fixes some set of manifolds, defines a suitable probability measure on this set and asks for the probability that a manifold has the given properties. 

A set up like this can be seen as a manifold analogue of the theory of random graphs (cf. \cite{Bol2}) and random simplicial complexes (cf. \cite{Kah}). Various models for random manifolds have been studied in dimensions two (cf. \cite{BIZ}, \cite{LS}, \cite{BM}, \cite{PS}, \cite{Mir}) and three (cf. \cite{DT}) and in general dimension (cf. \cite{ABBGNRS}). 

In this paper, we will be interested in the combinatorial model for random surfaces as introduced by Brooks and Makover in \cite{BM}. The idea of this model is to randomly glue together $2N$ triangles along their sides in such a way that an oriented surface without boundary is obtained. Many results are known about this model. 

The first natural problem is to understand the topology of the resulting surface. Because the number of triangles is given and the surfaces will be connected with probability tending to $1$ for $N\rightarrow\infty$, this is equivalent to knowing the distribution of the genus (which we shall denote by $g$). It turns out that the expected value of the genus behaves like:
\begin{equation*}
\ExV{N}{g} \sim \frac{N}{2}
\end{equation*}
for $N\rightarrow \infty$. This was proven with various degrees of precision by Gamburd and Makover in \cite{GM}, Brooks and Makover in \cite{BM}, Pippenger and Schleich in \cite{PS}, Gamburd in \cite{Gam} and Dunfield and Thurston in \cite{DT}. In particular, in \cite{Gam}, Gamburd exhibits the full asymptotic distribution of the genus.

These surfaces can be made geometric by defining a metric on the triangle. In \cite{BM}, Brooks and Makover study metrics that come from identifying the triangles with ideal hyperbolic triangles (glued without shear). A particularly nice feature of this choice of metric is that after compactifying the surfaces by adding points in the cusps one obtains a set of surfaces that is dense in any moduli space of closed surfaces. This follows from a theorem of Bely\v{\i} \cite{Bel}. This also makes it a natural problem to compare the behaviour of this combinatorial model for random surfaces to that of a models given by measures of finite total mass on moduli space, like for instance the Weil-Petersson measure as studied by Mirzakhani \cite{Mir}.

Brooks and Makover exhibit various geometric properties of these surfaces. Among these, they show, using a theorem of Brooks \cite{Bro}, that for any $\varepsilon>0$ with probability tending to $1$ for $N\rightarrow\infty$ the metric on the punctured surface and the compactified surface are $\varepsilon$-close in a suitable sense.

Another option is to endow the triangles with an equilateral Euclidean metric. In \cite{GPY}, Guth, Parlier and Young study the behaviour of pants decompositions of surfaces endowed with this metric. They show:
\begin{equation*}
\Pro{N}{\text{The surface has a pants decomposition of total length } \leq N^{7/6-\varepsilon}} \rightarrow 0
\end{equation*}
for $N\rightarrow \infty$ for any $\varepsilon >0$. They also show that in an appropriate sense, the same holds for the total pants length of hyperbolic surfaces in the model coming from the Weil-Petersson metric.

The length of short curves of random surfaces was investigated by the author in \cite{Pet}. It turns out that in the case of ideal hyperbolic triangles the expected value of the systole converges to a constant for $N\rightarrow\infty$ (approximately 2.48), both in the compactified and non-compactified case. Furthermore, in a more general setting of Riemannian metrics on the triangle, including the equilateral Euclidean metric, the lim sup and lim inf of the expected value of the systole are bounded from above and below respectively by constants depending only on the given metric on the triangle. Finally, the article describes the probability distribution of any fixed finite part of the length spectrum.

The aim of this paper is to investigate the dependence of the length spectrum of a random surface on the genus. That is, we want to compute conditional probabilities of the type:
\begin{equation*}
\ProC{N}{\text{the surface has }k\text{ closed curves of length }x}{ g\in D_N}
\end{equation*}
for some fixed $k\in\mathbb{N}$ and $x\in (0,\infty)$ and some sequence of sets $D_N\subset\mathbb{N}$. In words, the question we study is: if we only consider surfaces of a certain topology then what does the length spectrum look like?

As we will explain in Section \ref{sec_topgeom}, the length of a curve is strongly related to the combinatorial type of this curve. Loosely speaking, the combinatorial type of a curve records in which direction the curve turns on every triangle is traverses. In the hyperbolic setting the length of a curve is completely determined by the combinatorial type and in the Riemannian setting, the combinatorial type of a curve gives rise to bounds on the length of that curve. So to understand the probability distribution of lengths, one needs to understand the probability distributions of the number of curves of fixed combinatorial types. Given a combinatorial type of curves $w$, the random variable $Z_{N,w}$ will count the number of curves of this type on a random surface of $2N$ triangles. Furthermore, $\abs{[w]}$ will denote the number of combinatorial types equivalent to $w$ in a certain sense and $\abs{w}$ will denote the number of triangles a curve of type $w$ needs to traverse.

The main results of this article are the following:
\begin{thmA} Let $W$ be a finite set of combinatorial types of curves not containing a curve that turns around a single vertex of the triangulation. If we restrict to random surfaces with genus $g\in D_N$ and this set of random surfaces is non-negligible, then:
$$
Z_{N,w}\rightarrow Z_{w} \text{ in distribution for }N\rightarrow\infty
$$
for all $w\in W$, where the limit has to be taken over all even $N$ and where:
\begin{itemize}[leftmargin=0.2in]
\vspace{-0.175in}
\item $Z_{w}:\mathbb{N}\rightarrow\mathbb{N}$ is a Poisson distributed random variable with mean $\lambda_{w}=\frac{\aant{[w]}}{2\abs{w}}$ for all $w\in W$.
\item The random variables $Z_{w}$ and $Z_{w'}$ are independent for all $w,w'\in W$ with $w\neq w'$.
\end{itemize}
\end{thmA}

This theorem can be interpreted as a result on the independance of probabilities: we look at all the possible restrictions on our random surfaces based on genus and, as long as we do not restrict to a set of surfaces that asymptotically has probability $0$, all the finite parts of the length spectrum always behave in the same way.

We also note that Theorem A is a conditional version of Theorem B from \cite{Pet}, which can be obtained from Theorem A by choosing the sequence of sets $D_N=\{0,1,\ldots,\frac{N+1}{2}\}$. These two theorems are also analogues of the theorem of Bollob\'as on the distribution of circuits in random regular graphs \cite{Bol1}.

The proof of Theorem A relies on multiple steps. First, we determine the distribution of the genus, when restricting to random surfaces that carry the curves in $W$ in a fixed labelled way. To do this, we identify random surfaces with random elements of symmetric groups and use the Diaconis-Shahshahani upper bound lemma from \cite{DS} in a similar way to the proof of the distribution of the genus by Gamburd \cite{Gam}. Then we use the non-negligibility of our set of random surfaces to determine the moments of $Z_{N,w}$. These turn out to be the same as those in the unrestricted case computed in \cite{Pet}, which implies that the limiting distributions are the same.

As we will explain in Section \ref{sec_thmb}, there are restrictions on the genus to which the proof of Theorem A cannot possibly be applied. Note that these are necessarily restrictions to negligible sets of random surfaces. 

On the other hand, if we restrict to random surfaces with $1$ puncture or equivalently $g=\frac{N+1}{2}$, which is also a negligible set of random surfaces, it does turn out that the limits do behave in the same way. For this restriction we clearly need to assume that $N$ is odd. So we have:
\begin{thmB}
Let $W$ be a finite set of combinatorial types of curves not containing a curve that turns around a single vertex of the triangulation. If we restrict to random surfaces of genus $g=\frac{N+1}{2}$, then we have: 
$$
Z_{N,w}\rightarrow Z_{w} \text{ in distribution for }N\rightarrow\infty
$$
for all $w\in W$, where the limit has to be taken over all odd $N$.
\end{thmB}

Despite the fact that the result is similar to that in Theorem A, the proof of Theorem B very different. We still use the method of moments, but the heart of the proof of Theorem A, involving the Diaconis-Shahshahani upper bound lemma, must be replaced. Instead we use different counting methods, still coming from the character theory of the symmetric group. These methods are similar to those used in Appendix 6 of \cite{BIZ}, which Bessis, Itzykson and Zuber attribute to J. M. Drouffe (see also Theorem B of \cite{Pen} for a similar computation). In principle this same method could be used to compute the probabilites with other restrictions on the genus. The beauty of the case of genus $\frac{N+1}{2}$ however is that the expressions involved simplify significantly, whereas in other cases the computation quickly gets out of hand.

Theorems A and B can for example be used to compute conditional probability distributions for the systole function $\sys:\Omega_N\rightarrow\mathbb{R}$. From hereon until the end of the introduction, we will fix our sets $D_N$ that restrict the genus. These will be assumed to be either sets such that the resulting set of random surfaces is non-negligible or the sets $D_N=\left\{\frac{N+1}{2}\right\}$. In the first case all limits have to be taken over even $N$ and in the second case over odd $N$.

For the hyperbolic case we need the sets: 
\begin{equation*}
A_k = \left.\verz{\text{words }w\text{ in }\left(\begin{array}{cc}1 & 1 \\ 0 & 1 \end{array}\right)\text{ and }\left(\begin{array}{cc}1 & 0 \\ 1 & 1 \end{array}\right)}{\tr{w}=k}\middle/ \sim\right.
\end{equation*}
for all $k\in\mathbb{N}$, where $\sim$ denotes an equivalence on words which will be defined in Section \ref{sec_topgeom}.

\begin{cor1} In the hyperbolic setting we have that for all $\varepsilon>0$ sufficiently small and all $k\in\mathbb{N}$:
$$
\lim\limits_{N\rightarrow\infty} \ProC{N}{\abs{\sys-2\cosh^{-1}\left(\frac{k}{2}\right)}<\varepsilon}{g\in D_N} = \left(\prod\limits_{[w]\in\bigcup\limits_{i=3}^{k-1}A_i} \exp\left(-\frac{\aant{[w]}}{2\abs{w}}\right)\right) \left(1-\prod\limits_{[w]\in A_k}\exp\left(-\frac{\aant{[w]}}{2\abs{w}}\right)\right)
$$
\end{cor1}

The expression on the right hand side looks rather complicated. However, it can easily be computed for low values of $k$. For higher values of $k$ it is not difficult to approximate it. We have:
\begin{equation*}
\left(\prod\limits_{[w]\in\bigcup\limits_{i=3}^{k-1}A_i} \exp\left(-\frac{\aant{[w]}}{2\abs{w}}\right)\right) \left(1-\prod\limits_{[w]\in A_k}\exp\left(-\frac{\aant{[w]}}{2\abs{w}}\right)\right) \leq \exp(-k+3)
\end{equation*}
Furthermore we note that in the punctured hyperbolic setting the systole can only take values of the form $2\cosh^{-1}\left(\frac{k}{2}\right)$ for $k\in\mathbb{N}$, so in this case we can forget about the $\varepsilon$.

Using exactly the same method one can obtain formulas for the probability distribution of the $n^{th}$ shortest closed curve for any finite $n\in\mathbb{N}$. However, these formulas do become longer with increasing $n$.

We can also ask what happens if we only consider hyperbolic surfaces with given bounds on the systole. For this we have the following corollary:

\begin{cor2} Let $D_N\subset\mathbb{N}$ for all $N\in\mathbb{N}$ be a sequence of subsets such that the probability $\Pro{N}{g\in D_N}$ converges for $N\rightarrow\infty$ and let $x\in (2\log((3+\sqrt{5})/2),\infty)$. Then in the hyperbolic setting we have: 
$$
\lim\limits_{N\rightarrow\infty} \ProC{N}{g\in D_N}{\sys \leq x} = \lim\limits_{N\rightarrow\infty} \Pro{N}{g\in D_N}
$$
and:
$$
\lim\limits_{N\rightarrow\infty} \ProC{N}{g\in D_N}{\sys \geq x} = \lim\limits_{N\rightarrow\infty} \Pro{N}{g\in D_N}
$$
\end{cor2}

This corollary in fact works for any finite part of the length spectrum (so for any choice of bounds on the first $n$ curves for any finite $n\in\mathbb{N}$), as long as the conditions do not single out an empty set of surfaces.

In the Riemannian setting we are able to obtain bounds. Given a metric $d:\Delta\times\Delta\rightarrow [0,\infty)$ on our topological triangle $\Delta$ coming from a Riemannian metric, we define:
\begin{equation*}
m_1(d) = \min\verz{d(s,s')}{s,s' \text{ opposite sides of a gluing of two copies of }(\Delta,d)\text{ along one side}}
\end{equation*}
and:
\begin{equation*}
m_2(d) = \max\verz{d\left(\frac{e_i+e_j}{2},\frac{e_k+e_l}{2}\right)}{i,j,k,l\in\{1,2,3\},\;i\neq j,\;k\neq l}
\end{equation*}

The second corollary of Theorems A and B is:
\begin{cor3} In the Riemannian setting we have:
$$
\lim\limits_{N\rightarrow\infty} \ProC{N}{\sys < m_1(d)}{g\in D_N} = 0
$$ 
and for all $x\in [0,\infty)$:
$$
\lim\limits_{N\rightarrow\infty} \ProC{N}{\sys \geq x}{g\in D_N} \leq 1-\sum\limits_{k=2}^{\floor{x/m_2(d)}}  \left(e^{-\sum\limits_{j=1}^{k-1}\frac{2^{j-1}-1}{j}} -e^{-\sum\limits_{j=1}^k\frac{2^{j-1}-1}{j}}\right)
$$
where the last limit has to be taken over even $N$ in the first case and over odd $N$ in the second case.
\end{cor3}

The first of the two statements above is very easy to obtain and does not need Theorems A and B. The second statement is sharp in the sense that for any prescribed $m_2(d)$ and fixed $x\in [0,\infty)$ we can find a Riemannian metric that comes arbitrarily close to the upper bound. The construction of this metric is the same as the one in \cite{Pet}, the idea is to define a metric that forces short curves to trace a fixed path on every triangle.

\pagebreak
The organization of this article is as follows:
\begin{itemize}[leftmargin=0.2in]
\item[-] In the next section we briefly explain the relation between random surfaces, random graphs and random elements of the symmetric group.
\item[-] In the section after that we explain how to deduce topological and geometric properties of the surface from the corresponding graph and group elements.
\item[-] The fourth section shows how to restrict to random surfaces containing a certain set of curves.
\item[-] Section five contains the proof of Theorem A.
\item[-] Section six contains the proof of Theorem B.
\item[-] The final section explains the proofs of the corollaries.
\end{itemize}

\subsection*{Acknowledgement}
The author thanks his doctoral advisor Hugo Parlier for many useful discussions and carefully reading earlier drafts of this paper. He would furthermore like to thank Jeff Brock for useful discussions and Federica Fanoni for carefully checking some of the proofs.

\section{Random surfaces, random graphs and the symmetric group}

A random surface is a random gluing of $2N$ triangles (where $N\in\mathbb{N}$) along their sides. To describe the situation rigorously, we define the probability space:
\begin{equation*}
\Omega_N = \left\{\text{Partitions of }\{1,2,\ldots,6N\}\text{ into pairs} \right\}
\end{equation*}
with probability measure $\mathbb{P}_N$ that is given by the counting measure. 

The random surface corresponding to $\omega\in \Omega_N$ is obtained by labeling the sides of $2N$ triangles by the numbers $1,2,\ldots,6N$ in such a way that the sides $1$,$2$ and $3$ correspond to one triangle, so do the sides $4,5$ and $6$, and so forth. Furthermore the cyclic order in these labelings should correspond to the orientation on the triangle. Topologically, there is a unique way to glue the triangles along their sides as prescribed by $\omega$ such that the resulting surface is oriented with orientation corresponding to the orientation on the triangles. This will be the surface $S(\omega)$. 

The (cubic) dual graph to the triangulation on $S(\omega)$ will be denoted $\Gamma(\omega)$. The orientation on $S(\omega)$ induces a cyclic order on the edges at every vertex through the right hand rule. This cyclic order on the edges is exactly the same as the cyclic order on the labelled sides of the triangles. Such a pair of a graph and a cyclic order on the edges at every vertex is sometimes called a (cubic) \emph{fatgraph}, \emph{ribbon graph} or \emph{oriented graph}. Note that this orientation also gives us a notion of turning left or right when traversing a vertex.

Random surfaces can also be described by random elements of symmetric groups. This is done by associating a permutation $\sigma\in \symm{6N}$ to the vertices of the corresponding random graph and a permutation $\tau\in \symm{6N}$ to the edges.

$\sigma$ labels the left hand turns at every vertex. So if a vertex has half edges $i_1$, $i_2$ and $i_3$ emanating from it and the left hand turns at this vertex are of the form $(i_1,i_2)$, $(i_2,i_3)$ and $(i_3,i_1)$ then we add the cycle $(i_1\;i_2\;i_3)$ to $\sigma$, as in Figure \ref{pic4} below:
\begin{figure}[H]
\begin{center} 
\includegraphics[scale=1]{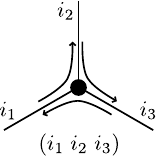} 
\caption{A $3$-cycle corresponding to a vertex. The arrows indicate the left hand turns.}
\label{pic4}
\end{center}
\end{figure}
So $\sigma$ is a product of $2N$ disjoint $3$-cycles. 

$\tau$ records which half edge is glued to which other half edge in the graph. If half edge $i_1$ is glued to half edge $i_2$ in the graph then we add a cycle $(i_1\; i_2)$ to $\tau$ as in Figure \ref{pic5} below:
\begin{figure}[H]
\begin{center} 
\includegraphics[scale=1]{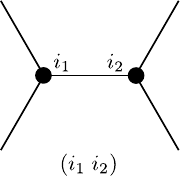} 
\caption{A $2$-cycle corresponding to an edge.}
\label{pic5}
\end{center}
\end{figure}
So $\tau$ is a product of $3N$ disjoint $2$-cycles. 

It is well known that the conjugacy class of an element $\pi\in\symm{6N}$ is determined by its cycle type. If $\lambda$ is a partition of $6N$ (we will sometimes abbreviate this to $\lambda\models 6N$) we will denote the corresponding conjugacy class in $\symm{6N}$ by $K(\lambda)$. So we have:
\begin{equation*}
\sigma\in K\left(3^{2N} \right),\; \tau\in K\left(2^{3N} \right)
\end{equation*}
Where $1^{i_1}2^{i_2}\ldots (6N)^{i_{6N}}$ denotes the partition of $6N$ with $i_1$ parts equal to $1$, $i_2$ parts equal to $2$, and so forth.

This means that we can identify the set of random surfaces with $K\left(3^{2N}\right)\times K\left(2^{3N} \right)$. Using the counting measure, we can turn this set into a probability space again. 

Note that this set is a lot larger than $\Omega_N$, we get many copies of every element of $\Omega_N$, corresponding to different choices of $\sigma$. For example, in $\Omega_N$ half-edges $1$, $2$ and $3$ always emanate from the same vertex on the corresponding graph. In $K\left(3^{2N}\right)\times K\left(2^{3N} \right)$ there could be a vertex whose half-edges are labelled $1$, $2$ and $7$. We could of course choose to fix $\sigma = (1\;2\;3)\ldots (6N-2\;6N-1\;6N)$ so that we get $\Omega_N$ back. 

However, these extra choices do not influence the topology or geometry (if we relabel $\tau$ along with $\sigma$), so from the point of view of random surfaces, the two probability measures are the same. Sometimes it is convenient to also randomly pick $\sigma$, so we will not fix it. Because from the point of view of random surfaces they are equivalent, we will denote both measures by $\mathbb{P}_N$. The surface corresponding to a pair of permutations $(\sigma,\tau)\in K\left(3^{2N}\right)$ will be denoted $S(\sigma,\tau)$ and the corresponding oriented graph will be denoted $\Gamma(\sigma,\tau)$.

\section{The topology and geometry of random surfaces} \label{sec_topgeom}

The topology of a random surface corresponding to a partition $\omega\in\Omega_N$ is determined by $LHT(\omega)$: the number of left hand turn cycles in the oriented graph $\Gamma(\omega)$. In fact, if $S(\omega)$ is connected, the genus of the surface corresponding to $\omega\in\Omega_N$ is given by:
\begin{equation*}
g(\omega) = 1 + \frac{N}{2} - \frac{LHT(\omega)}{2}
\end{equation*}
$LHT(\omega)$ in turn is equal to the number of disjoint cycles in $\sigma\tau$ for any choice of $(\sigma,\tau)\in K\left(3^{2N}\right)\times K\left(2^{3N} \right)$ corresponding to $\omega$. This is because the permutation $\sigma\tau$ describes what happens to a given half edge after consecutively traversing one edge and then taking a left hand turn. As such, we will sometimes write $\mathrm{LHT}(\sigma,\tau)$ for the number of disjoint cycles in $\sigma\tau$.

To formalise the notion of non-negligibility mentioned in the introduction we will need the following:
\begin{dff} A sequence of subsets $D_N\subset\mathbb{N}$ for $N\in\mathbb{N}$ will be called \emph{non-negligible with respect to the genus} if:
$$
\liminf\limits_{N\rightarrow\infty} \Pro{N}{g\in D_N} >0
$$
\end{dff}

Besides the topology of the surface itself we will also need to control the topology of curves on the surfaces. We will need the following theorem about separating curves (Theorem C in \cite{Pet}):
\begin{thm}\label{thm_sepcurves} \cite{Pet}
Let $C\in (0,1)$. We have:
$$
\lim\limits_{N\rightarrow\infty} \Pro{}{\substack{\text{A random cubic graph on }2N \text{ vertices contains}\\ \text{a separating circuit of }\leq C\log_2(N)\text{ edges }}} = 0
$$ 
\end{thm}

We want to study the geometry of curves on random surfaces. To do this, we will need the following two $2\times 2$ matrices:
\begin{equation*}
L=\left(
\begin{array}{cc}
1 & 1 \\
0 & 1
\end{array}
\right)
\text{ and }
R=\left(
\begin{array}{cc}
1 & 0 \\
1 & 1
\end{array}
\right)
\end{equation*}
The set of all words in $L$ and $R$ will be denoted $\{L,R\}^*$. Elements in this set will sometimes be interpreted as matrices and sometimes as strings in two letters, it will be clear from the context which of the two is the case. We need to define the following equivalence relation on this set:
\begin{dff}
Two words $w\in\{L,R\}^*$ and $w'\in\{L,R\}^*$ will be called \emph{equivalent} if one of following two conditions holds:
\vspace{-0.175in}
\begin{itemize}[leftmargin=0.2in]
\item[-] $w'$ is a cyclic permutation of $w$
\item[-] $w'$ is a cyclic permutation of $w^*$, where $w^*$ is the word obtained by reading $w$ backwards and replacing every $L$ with an $R$ and vice versa.
\end{itemize}
If $w\in\{L,R\}^*$, we will use $[w]$ to denote the set of words equivalent to $w$.
\end{dff}

The reason this has anything to do with the geometry of curves on random surfaces is the following. We first look at a random surface triangulated by ideal hyperbolic triangles. To properly define such a gluing we need one extra parameter per pair of sides in the gluing called the shear of the gluing. This parameter measures the signed distance between the midpoints of the two sides. Here the midpoint of a side of a triangle is determined by where the orthogonal from the corner opposite this side hits the side, as illustrated in Figure \ref{pic8} below:
\begin{figure}[H]
\begin{center} 
\includegraphics[scale=1.5]{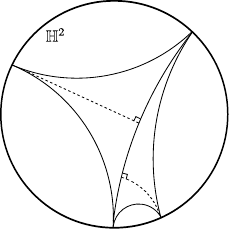} 
\caption{Shear along a common side of two triangles in the Poincar\'e disk model of the hyperbolic plane.}
\label{pic8}
\end{center}
\end{figure}
The sign of the shear can be defined using the orientation on the surface. In this article all gluings will have shear coordinate $0$ at every pair of sides. So this means that the triangles will always be glued such that the orthogonals from the two corners opposite a side meet each other.

Given an essential closed curve $\gamma$ on such a surface, it is a classical fact from hyperbolic geometry that it is homotopic to a unique closed geodesic $\tilde{\gamma}$. If we trace this geodesic and record whether it turns left or right at every triangle it passes (which is well defined by the orientation on the surface) this gives a word $w_{\tilde{\gamma}}\in \{L,R\}^*$. Note however that this word is only defined up to the equivalence defined above. The resulting equivalence class of words is what we called the combinatorial type in the introduction.

It is another classical fact from hyperbolic geometry that the length of $\tilde{\gamma}$ is given by:
\begin{equation*}
\ell(\tilde{\gamma}) = 2\cosh^{-1}\left(\frac{\tr{w_{\tilde{\gamma}}}}{2}\right)
\end{equation*}
This implies that the number of curves on the punctured surface of a fixed length is given by the number of appearances of all the possible words in $L$ and $R$ with the corresponding trace. This leads us to the following definition:

\begin{dff}
Let $N\in\mathbb{N}$ and $w\in\{L,R\}^*$. Define $Z_{N,[w]}:\Omega_N\rightarrow \mathbb{N}$ by:
$$
Z_{N,[w]}(\omega) = \aant{\verz{\gamma}{\gamma\text{ a circuit on }\Gamma(\omega)\text{, }\gamma\text{ carries }w}}
$$
where $\Gamma(\omega)$ is the dual graph to the triangulation corresponding to $\omega\in\Omega_N$.
\end{dff}
A circuit in this definition is a closed path on $\Gamma(\omega)$ that meets all of its vertices and edges at most once.

So if we understand the probability distribution of the random variables $Z_{N,[w]}$ for all $w\in\{L,R\}^*$ we understand the probability distribution of the length spectrum of the punctured random surfaces.

One can also compactify the surfaces by adding points in the cusps. This compactifications goes as follows. Our surfaces carry a conformal structure. In this structure we can choose neighborhoods of the cusps that are conformally equivalent to punctured disks in $\mathbb{C}$. These disks can be compactified by adding a point in the puncture. This gives us a new conformal structure on a now closed surface. Using the Uniformization Theorem there is a unique complete hyperbolic structure that is conformally equivalent to this new conformal structure. 

Curves on the original surface project to curves on the compactified surface. However, homotopically essential curves on the original surface might project to trivial curves. For a curve $\gamma$, we will write $\gamma\nsim_C 0$ to indicate that $\gamma$ is non-null homotopic on the compactified surface. Theorem \ref{thm_sepcurves} above will be used to control essential curves on the punctured surface that project to trivial curves on the closed surface.

The lengths of all the curves can also change. To solve that problem, we need information on how these lengths change. In \cite{Pet} we obtained the following result (Proposition 5.4, which is a sharper version of Theorem 2.1 of \cite{BM}):
\begin{prp}\cite{Pet}\label{prp_cusplength} Let $\varepsilon >0$. We have:
$$
\Pro{N}{\text{There exists a simple closed curve }\gamma \nsim_C 0\text{ with } \frac{\ell_O(\gamma)}{\ell_C(\gamma)} < \frac{1}{1+\varepsilon} \text{ or } \frac{\ell_O(\gamma)}{\ell_C(\gamma)}> 1+\varepsilon} = 1-\mathcal{O}(N^{-1})
$$
for $N\rightarrow\infty$, where $\ell_O$ and $\ell_C$ denote the length of a curve on the punctured and compactified surface respectively.
\end{prp}

Furthermore, we have the following lemma by Brooks:
\begin{lem}\label{lem_Brooks}\cite{Bro} For $N\in\mathbb{N}$ sufficiently large, there is a constant $\delta(N)$ with the following property: If $\omega\in\Omega_N$ such that the corresponding punctured surface $S_O(\omega)$ has $1$ cusp. Then for every geodesic $\gamma$ in the compactified surface $S_C(\omega)$ there is a geodesic $\gamma'$ in $S_O(\omega)$ such that the image of $\gamma'$ is homotopic to $\gamma$, and:
$$
\ell(\gamma) \leq \ell(\gamma') \leq (1+\delta(N))\ell(\gamma)
$$
Furthermore, $\delta(N)\rightarrow 0$ as $N\rightarrow\infty$.
\end{lem}
In fact, the lemma by Brooks is more general, but the above statement is all we will need.

The proposition and lemma above imply that the same random variables $Z_{N,[w]}$ also determine the length spectrum of the compactified hyperbolic random surfaces.

For metrics on random surfaces coming from a Riemannian metric on the triangle we do not have the nice combinatorial description of lengths. However, we can still get estimates on the probability distribution of the length spectrum, depending only on the metric on the triangle, if we understand the probability distribution of the random variables defined above.

So, in the end we are interested in the restrictions of the random variables $Z_{N,[w]}$ to some subsets of $\Omega_N$ (or $K\left(3^{2N}\right)\times K\left(2^{3N}\right)$ equivalently) defined by conditions on the genus. If the restriction is determined by the fact that the genus $g\in D_N$ then we denote the corresponding restricted random variable:
\begin{equation*}
\left(Z_{N,[w]}\right)\vert_{g\in D_N} : \verz{\omega\in\Omega_N}{g\in D_N}\rightarrow\mathbb{N}
\end{equation*}

\section{Restricting to surfaces carrying a fixed set of curves} \label{sec_restrict}

The goal of this article is to understand the relation between the distribution of the number of words of a fixed type and that of the genus. In particular, we want to be able to restrict to sets of random surfaces of a certain genus and then count how many circuits (recall that by `circuit' we mean a closed path in a graph that visits each of its vertices and edges only once) carrying a fixed word in $L$ and $R$ the surfaces with this genus have.

However, what we will actually do is restrict to surfaces carrying a certain set of words as circuits. These circuits will initially be labelled by numbers in $\{1,2,\ldots, 6N\}$ in a fixed way. We will then study the distribution of the genus under the condition that the surface contains these curves. Once we have determined these distributions, we shall `invert' them in order to obtain the distributions of the number of appearances of fixed words (now unlabelled) with conditions on the genus.

We will mainly work with the description of random surfaces by elements of the symmetric group. In what follows we explain how to restrict to elements of the symmetric group describing random surfaces containing a fixed labelled set of words. This process relies on the following observation.

\begin{obs}
Suppose a random surface contains a word in $L$ and $R$ that is represented by a circuit. This means that the surface contains an annulus like the one in Figure \ref{pic_repl1} below:
\begin{figure}[H]
\begin{center} 
\includegraphics[scale=1.25]{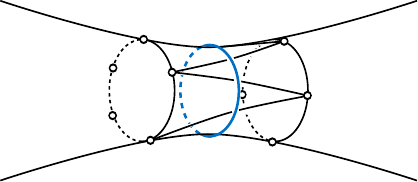} 
\caption{A subsurface corresponding to a circuit representing a word in $L$ and $R$.}
\label{pic_repl1}
\end{center}
\end{figure}

We want to compute the genus of the surface in Figure \ref{pic_repl1}, hence we need to count the number of vertices in the triangulation. Before we do this, we remove the subsurface around the blue curve above and replace it with two polygons as in Figure \ref{pic_repl2} below:
\begin{figure}[H]
\begin{center} 
\includegraphics[scale=1.25]{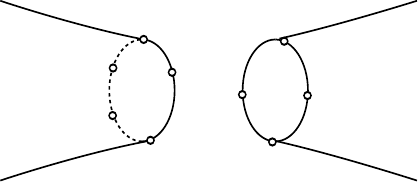} 
\caption{Cutting out the subsurface and replacing it by a $4$ and a $5$-gon.}
\label{pic_repl2}
\end{center}
\end{figure}

When we do this, we obtain a new surface which is `triangulated' by some number of triangles and the two polygons. The topology of the new surface is of course different from that of the original surface. The number of vertices of the triangulation however stays the same (unless the word corresponding to the circuit consists of only $L$'s or only $R$'s, which we will assume not to be the case). So, if we are only interested in the number of vertices of the triangulation, we can just as well count the vertices on the surface of Figure \ref{pic_repl2}.

Note that in Figure \ref{pic_repl1} we have drawn a triangulated circuit that is `properly embedded'. That is, besides the sides that form the circuit no other sides of its triangles are paired. In a general triangulated surface this need not be the case. However, this will not cause us problems, we just need to pair the corresponding sides of the polygons as well. Figure \ref{pic_replalt} shows an example:
\begin{figure}[H]
\begin{center} 
\includegraphics[scale=1]{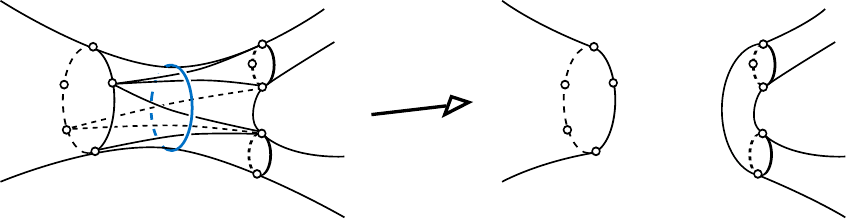} 
\caption{Replacing a circuit in which two of its triangles share another side.}
\label{pic_replalt}
\end{center}
\end{figure}
We have not drawn a complete triangulation of the circuit above on the left, some of the triangles on the back of the circuit are missing. We have done this in order not to clutter the image too much.

\end{obs}

We finally remark the following three things:
\vspace{-0.2cm}
\begin{itemize}
\item[-]  We have already seen that if the circuit represents a left hand turn path, the number of vertices does change after removing the corresponding subsurface. So for what follows we will assume that the circuit contains at least one left hand turn and one right hand turn.
\item[-] It is essential to assume that the word in $L$ and $R$ we replace is represented by a circuit. In particular, what we will describe does not work for words represented by general cycles (closed walks that might visit vertices or edges multiple times). It turns out that the probability that these appear on a random surface tends to $0$, so we will not lose anything in the length spectrum by restricting to circuits.
\item[-] On the level of the dual graph this replacing a subsurface by polygons comes down to replacing an oriented circuit by two oriented stars with some number of emanating half edges.
\end{itemize}

We will now turn the idea from the observation above into a general description. The set up will be as follows. We have a finite set $W\subset \{L,R\}^*/\sim$ and $m\in\mathbb{N}^W$. Using this data we define a labelled graph $\Gamma(W,m)$ of
\begin{equation*}
\sum_{w\in W}m_w\abs{w}
\end{equation*}
vertices, that is a disjoint union of
\begin{equation*}
\sum_{w\in W}m_w
\end{equation*}
circuits such that for every $w\in W$ there are $m_w$ circuits carrying $w$. We now want to understand the distribution of the genus in the probability space:
\begin{equation*}
\Omega_N(W,m)=\verz{(\sigma,\tau)\in K\left(3^{2N}\right)\times K\left(2^{3N}\right)}{\Gamma(W,m)\subset\Gamma(\sigma,\tau)}
\end{equation*}
The inclusion `$\Gamma(W,m)\subset\Gamma(\sigma,\tau)$' above is meant as an inclusion as oriented graphs. That is to say, the words induced by the orientation of $\Gamma(\sigma,\tau)$ on the circuits of $\Gamma(W,m)$ are the words in $W$ with multiplicities $m$. So this means that we want to understand the genus distribution in set of random surfaces that carry the words in $W$ with multiplicities $m$ in a fixed labelled way.

Using the observation above, we will describe a probability space $\Omega_N'(W,m)$ of random surfaces and a map:
\begin{equation*}
P_{W,m}:\Omega_N(W,m)\rightarrow \Omega_N'(W,m)
\end{equation*}
that preserves the number of vertices. This map will not be injective. However, every point in $\Omega_N'(W,m)$ will have the same number of preimages. This implies that the distribution of the number of vertices of a random surface is the same in both spaces.   

We now start by describing $\Omega_N'(W,m)$. This should be the space of random surfaces out of $\Omega_N(W,m)$ in which $\Gamma(W,m)$ has been replaced by polygons. Each circuit should contribute two polygons. The sizes of these two polygons can be read of from the $L$'s and the number of $R$'s in the corresponding word. Concretely, if $w$ contains $l_w$ left hand turns and $r_w$ right hand turns then the corresponding circuit will be replaced by an $l_w$- and an $r_w$-gon\footnote{Technically $l_w$ and $r_w$ are not well defined on the equivalence class $w$, only the unordered pair $\{l_w,r_w\}$ is. We can solve this issue by choosing a representative of every $w\in W$ as part of our data.}. So, that means that $\Omega_N'(W,m)$  should be the probability space of gluings of $m_w$ $l_w$-gons, $m_w$ $r_w$-gons for all $w\in W$ and $2N-\sum_{w\in W}m_w\abs{w}$ triangles. 

In terms of the symmetric group the left hand turns in an $n$-gon can be described by an $n$-cycle. This is just a generalization of the description of the left hand turns in a triangle in terms of a $3$-cycle in the symmetric group. The side pairings in such a gluing can still be described by a product of $2$-cycles. This means that we set:
\begin{equation*}
\Omega_N'(W,m)=K\left(3^{2N-\sum\limits_{w\in W} m_w \abs{w}}\cdot\prod_{w\in W}l_w^{m_w}\cdot\prod_{w\in W}r_w^{m_w}\right)\times K\left(2^{3N-\sum\limits_{w\in W} m_w \abs{w}} \right)
\end{equation*}
To lighten notation, we shall sometimes denote these two conjugacy classes by $K_3(W,m)$ and $K_2(W,m)$ respectively.

From the reasoning above, we see that if we take a surface in $\Omega_N(W,m)$ and replace $\Gamma(W,m)$ by polygons then we obtain an element in $\Omega'_N$.\footnote{In principle we also need to redefine the labels, but without loss of generality we can assume that the labels we have used for $\Gamma(W,m)$ are $\{6N+1-2\sum\limits_{w\in W}m_w\aant{w},\ldots, 6N \}$ in which case we do not need to relabel.} In other words, this process defines a map:
\begin{equation*}
P_{W,m}:\Omega_N(W,m)\rightarrow \Omega_N'(W,m)
\end{equation*}

We have also seen that the topology of the resulting surface will in general be different. The number of vertices in the triangulation of this surface will however be equal to the number of vertices in the triangulation of the original surface. We also note that if $(\sigma',\tau')\in\Omega_N'(W,m)$ then the number of vertices of the corresponding triangulation is still given by the number of cycles in $\sigma'\tau'$.

What we still need to show is that the map $P_{W,m}$ behaves well with respect to the probability measures. This will follow from the following lemma:
\begin{lem}\label{lem_rest}
The map $P_{W,m}:\Omega_N(W,m)\rightarrow \Omega_N'(W,m)$ is surjective and the number $\aant{P_{W,m}^{-1}(\sigma',\tau')}$ depends only on $W$,$m$ and $N$ and not on $(\sigma',\tau')$.
\end{lem}

\begin{proof}
Surjectivity follows from `reconstructing' $(\sigma,\tau)$ from $(\sigma',\tau')$. In order to obtain a pair $(\sigma,\tau)$ such that $P_{W,m}(\sigma,\tau)=(\sigma',\tau')$ we need to remove polygons from the corresponding surface and glue the subsurface corresponding to $\Gamma(W,m)$ back in. The possibility of doing this depends on the sizes of the polygons out of which the surface is built. By construction of $\Omega_N'(W,m)$ we can find polygons of the right sizes in our surface. Because everything is labelled, this immediately gives us a pair $(\sigma,\tau)$ and by construction we have $\Gamma(W,m)\subset\Gamma(\sigma,\tau)$, which implies that $(\sigma,\tau)\in\Omega_N(W,m)$.

The number $\aant{P_{W,m}^{-1}(\sigma',\tau')}$ depends on the number of choices we have in reconstructing $(\sigma,\tau)$. First of all we need to choose polygons of the right sizes to remove. There might be multiple choices in this, especially when there are words in $W$ with exactly $3$ $L$'s or $3$ $R$'s in them. After removing these polygons we need to choose which hole is glued to which circuit. If we have multiple words with the same numbers of $L$'s or $R$'s in them then we can choose any of them to glue to a particular hole of that size. Finally, once we have decided which hole is glued to which side of which circuit, we need to decide how these are glued. Because of the orientation there are restrictions on this choice. In particular, once we have fixed one pair of sides of triangles that is to be identified, the choice is fixed. Because all the choices above do not depend on $(\sigma',\tau')$, the number  $\aant{P_{W,m}^{-1}(\sigma',\tau')}$ is constant on $\Omega_N'(W,m)$, which proves the lemma.
\end{proof}

The following proposition follows immediately:
\begin{prp}\label{prp_rest} Let $N\in \mathbb{N}$, $W\subset\{L,R\}^*/\sim$ be a finite set and $m\in \mathbb{N}^W$. Furthermore, let $\Gamma(W,m)$ be an oriented labelled graph representing $W$ as a disjoint union of circuits. Then for all $k\in\mathbb{N}$:
$$
\ProC{N}{\mathrm{LHT}=k}{\Gamma(W,m)\subset \Gamma}=\frac{\aant{\verz{(\sigma',\tau')\in K_3(W,m)\times K_2(W,m)}{\sigma'\tau'\text{ has }k\text{ cycles}}}}{\aant{K_3(W,m)\times K_2(W,m)}}
$$
\end{prp}

\begin{proof} We will write $C=\aant{P_{W,m}^{-1}(\sigma',\tau')}$ for any (and by Lemma \ref{lem_rest} all) $(\sigma',\tau')\in K_3(W,m)\times K_2(W,m)$. Furthermore we write $\mathrm{LHT}(\sigma',\tau')$ for the number of cycles in $\sigma'\tau'$. 

Because $P_{W,m}$ preserves corners we have:
\begin{align*}
\ProC{N}{\mathrm{LHT}=k}{\Gamma(W,m)\subset \Gamma} & = \frac{\aant{\verz{(\sigma,\tau)\in K\left(3^{2N}\right)\times K\left(2^{3N}\right)}{\substack{\displaystyle{\Gamma(W,m)\subset\Gamma(\sigma,\tau)} \\ \displaystyle{\text{and }\mathrm{LHT}(\sigma,\tau)=k}}}}}{\aant{\verz{(\sigma,\tau)\in K\left(3^{2N}\right)\times K\left(2^{3N}\right)}{\Gamma(W,m)\subset\Gamma(\sigma,\tau)}}} \\[3mm]
 & = \frac{\aant{\verz{(\sigma,\tau)\in K\left(3^{2N}\right)\times K\left(2^{3N}\right)}{\substack{\displaystyle{\Gamma(W,m)\subset\Gamma(\sigma,\tau)} \\ \displaystyle{\text{and }\mathrm{LHT}\left(P_{W,m}(\sigma,\tau)\right)=k}}}}}{\aant{\verz{(\sigma,\tau)\in K\left(3^{2N}\right)\times K\left(2^{3N}\right)}{\Gamma(W,m)\subset\Gamma(\sigma,\tau)}}}
\end{align*}
Now we apply Lemma \ref{lem_rest} to obtain that the numerator above is equal to:
\begin{equation*}
C\cdot\aant{\verz{(\sigma',\tau')\in K_3(W,m)\times K_2(W,m)}{\mathrm{LHT}(\sigma',\tau')=k}}
\end{equation*}
and the denominator to:
\begin{equation*}
C\cdot\aant{K_3(W,m)\times K_2(W,m)}
\end{equation*}
When we fill these in we obtain the proposition.
\end{proof}

To further shorten notation, we will write:
\begin{equation*}
M=M(W,m)=\sum\limits_{w\in W} m_w\abs{w}
\end{equation*}¨
From hereon we will also identify $(\sigma',\tau')$ with $(\sigma,\tau)$. That is to say, we will generally consider both $\sigma$ and $\tau$ as elements in $\symm{6N-2M}$.

From this point on, the proofs of Theorem A and Theorem B go in different directions. We will start with Theorem A.

\section{The proof of Theorem A}

\subsection{The proof strategy}
The main part of the proof of Theorem A relies on the Diaconis-Shahshahani upper bound lemma. This is an upper bound on the total variational distance between a probability measure on a finite group and the uniform measure on that same group. Before we can state this lemma, we need to introduce some notation.
\begin{dff} Let $G$ be a finite group. 
\vspace{-0.175in}
\begin{itemize}[leftmargin=0.2in]
\item[-] The set of irreducible unitary representations of $G$ will be denoted $\hat{G}$.
\item[-] For $\rho\in \hat{G}$ and a probability measure $\mathbb{P}$ on $G$. The Fourier transform of $\mathbb{P}$ at $\rho$ is the linear map:
         $$
         \hat{\mathbb{P}}(\rho) = \sum\limits_{g\in G} \mathbb{P}[g]\rho(g)
         $$
\end{itemize}
\end{dff}
We also need to define the total variational distance between two probability measures on a finite set.
\begin{dff}
Let $S$ be a finite set and $\mathcal{P}(S)$ its power set. Furthermore, let $f,g:\mathcal{P}(S)\rightarrow\mathbb{C}$. Then the total variational distance between $f$ and $g$ is given by:
\begin{equation*}
\norm{f-g}=\max\verz{\abs{f(A)-g(A)}}{A\in\mathcal{P}(S)}
\end{equation*}
\end{dff}

The Diaconis-Shahshahani upper bound lemma is the following:
\begin{lem}\label{lem_DS}\cite{DS} Let $G$ be a finite group. Furthermore, let $\mathbb{P}$ be a probability measure on $G$ and let $\mathbb{U}_G$ denote the uniform probability measure on $G$ then:
$$
\norm{\mathbb{P}-\mathbb{U}_G}^2 \leq \frac{1}{4}\sum\limits_{\substack{\rho\in \hat{G} \\ \rho\neq \mathrm{id}}} \dim(\rho) \tr{\hat{\mathbb{P}}(\rho) \overline{\hat{\mathbb{P}}(\rho)}}
$$
\end{lem}

In \cite{Gam}, Gamburd applies this lemma to the distribution of $\sigma$ and $\tau$ as elements of the alternating group. It turns out that this switch to the alternating group is essential, one has to avoid the sign representation of the symmetric group. Unfortunately, in our case $\sigma$ and $\tau$ do not generally lie in $\alt{6N-2M}$. We need the following lemma:

\begin{lem}\label{lem_alt}
Let $N$ be even. If $\sigma \in K_3(W,m)$ and $\tau\in K_2(W,m)$ then $\sigma\tau\in \alt{6N-2M}$
\end{lem}

\begin{proof}
We have: 
\begin{align*}
\sigma \notin \alt{6N-2M} & \Leftrightarrow & \sigma \text{ contains an odd number of even cycles} \\
   & \Leftrightarrow & W \text{ contains an odd number of words }w\text{ with }l_w\text{ odd} \\
   & & \text{and }r_w\text{ even, or vice versa} \\
   & \Leftrightarrow &  W\text{ contains an odd number of words of odd word length} \\
   & \Leftrightarrow &  \sum\limits_{w\in W}m_w \abs{w} \text{ is odd} \\
   & \Leftrightarrow & 3N-\sum\limits_{w\in W}m_w \abs{w} \text{ is odd} \\
   & \Leftrightarrow &  \tau \notin \alt{6N-2M}
\end{align*}
So either $\sigma$ and $\tau$ are both elements in $\alt{6N-2M}$ in which case their product is as well, or $\sigma$ and $\tau$ both have negative sign, in which case their product is also an element of $\alt{6N-2M}$.
\end{proof}

This lemma implies that the probability measure of the product $\sigma\tau$ can be seen as a probability measure on $\alt{6N-2M}$ when $N$ is even. We will denote the probability measure by $\mathbb{P}_{3\star 2,N,W,m}$ and we have:
\begin{equation*}
\mathbb{P}_{3\star 2,N,W,m} = \mathbb{P}_{3,N,W,m} \star \mathbb{P}_{2,N,W,m}
\end{equation*}
as a measure on $\symm{6N-2M}$, where $\mathbb{P}_{3,N,W,m}$ and $\mathbb{P}_{2,N,W,m}$ are the uniform probability measures on $K_3(W,m)$ and $K_2(W,m)$ repsectively and $\star$ denotes the convolution product. 

Because of the lemma above we will assume that $N$ is even for the remainder of this section.

$\Gamma_{(W,m)}$ will denote a fixed labelled representation as disjoint circuits of the words $W$ with multiplicities $m\in\mathbb{N}^W$. It follows from Proposition \ref{prp_rest} that for such a representation we have: 
\begin{equation*}
\ProC{N}{\sigma\tau\text{ has }k\text{ cycles}}{\Gamma_{(W,m)}\subset\Gamma}  = \Pro{3\star 2,N,W,m}{\sigma\tau\text{ has }k\text{ cycles}}
\end{equation*}

We want to prove the following analogue of Gamburd's Theorem 4.1 \cite{Gam}:
\begin{thm}\label{thm_limit}
Let $W$ be a finite set of words and $m\in \mathbb{N}^W$ then:
$$
\lim\limits_{N\rightarrow\infty}\norm{\mathbb{P}_{3\star 2,N,W,m} - \mathbb{U}_{N,W,m}} = 0
$$
where $\mathbb{U}_{N,W,m}$ denotes the uniform probability measure on $\alt{6N-2M}$.
\end{thm}

This will imply Theorem A, as we shall explain in the last part of this section.

\subsection{Representations of the symmetric group}\label{sec_charsn} From the previous subsection it is clear that we need to gather some facts about the irreducible characters of the alternating group, which are closely related to those of the symmetric group. We deal with the symmetric group in this subsection and with the alternating group in the next subsection. We will only gather the facts we need for the proof of Theorem \ref{thm_limit}. For a comprehensive treatment of the representation theory of the symmetric and alternating group we refer the reader to \cite{dBR} and \cite{JK}. In what follows we forget about triangles for a moment and $N$ will just be a natural number. 

As for any finite group, the irreducible representations of $\symm{N}$ are in bijection with the conjugacy classes of $\symm{N}$. The nice feature of $\symm{N}$ is that there is a natural bijection between these two sets.

We recall that the conjugacy classes of $\symm{N}$ are labelled by partitions $\lambda\models N$. Such a partition can be represented by what is called a Young diagram. If $\lambda=(\lambda_1,\lambda_2,\ldots,\lambda_k)$ then the corresponding Young diagram is formed by $k$ left aligned rows of boxes where row $i$ has length $\lambda_i$. For example, if $\lambda=(4,4,3,1)$ then the corresponding Young diagram is:
\begin{equation*}
\yng(4,4,3,1)
\end{equation*}
A filling of such a diagram with the numbers $1,2,\ldots,N$ is called a Young tableau. In general we will not make a distinction between a partition, its corresponding Young diagram or a Young tableau corresponding to that. A permutation acts on a Young tableau by permuting the numbers in its boxes. So for example, we have:
\begin{equation*}
\Yvcentermath1 (1\;4\;2) \cdot\; \young(12,35,4) =  \young(41,35,2)
\end{equation*}
Note that a permutation does not change the shape of the tableau. 

One obtains a representation for every $\lambda\models N$ by taking the Young tableaux of shape $\lambda$ as a basis for a \linebreak $\mathbb{C}$-vector space and extending the action linearly. In general these representations are not irreducible, but there is a procedure to obtain exactly one distinct irreducible representation as a subrepresentation of each of them. We will denote the corresponding vector space $V^\lambda$, the representation $\rho^\lambda:\symm{N}\rightarrow \mathrm{GL}(V^\lambda)$, its character $\chi^\lambda:\symm{N}\rightarrow\mathbb{C}$ and its dimension $f^\lambda=\dim\left(V^\lambda\right)$. By construction these representations form the complete set of irreducible representations of $\symm{N}$.

To apply the Diaconis-Shahshahani upper bound lemma we need to obtain bounds on the dimension and the characters of all the irreductible representations. For the dimension we need the notion of the hooklength $h(b)$ of a box $b\in\lambda$. This is simply $1$ plus the number of boxes to the right of $b$ plus the number of boxes below $b$. As an example the tableau below is filled with the hooklengths of the corresponding boxes:
\begin{equation*}
\young(6421,31,1)
\end{equation*}

We have the following classical theorem for the numbers $f^\lambda$ (see for instance Equation 2.37 on page 44 of \cite{dBR}):
\begin{thm}\label{thm_hook} \emph{(Hook length formula)} Let $\lambda\models N$. Then:
$$
f^\lambda = \frac{N!}{\prod\limits_{b\in\lambda} h(b)}
$$
\end{thm}

Next, we need to gather some facts about the characters $\chi^\lambda$. The first one is the Murnaghan-Nakayama rule (Lemma 4.15 and Equation 4.21 on pages 77 and 78 of \cite{dBR}):
\begin{thm} \label{thm_MN}\emph{(Murnaghan-Nakayama rule)} Let $g\in\symm{N}$ be such that:
$$
g=hc
$$
where $h\in\symm{N-m}$ and $c$ an $m$-cycle. Then:
$$
\chi^\lambda (g) = \sum_\mu (-1)^{r(\lambda,\mu)}\chi^\mu(h)
$$
where the sum above runs over all tableaux $\mu$ that can be obtained from $\lambda$ by removing a continuous region on the boundary of $\lambda$ consisting of $m$ boxes (called a \emph{skew $m$ hook} or a \emph{rim hook}). And $r(\lambda,\mu)$ is the number of rows in the skew $m$ hook that needs to be removed from $\lambda$ to obtain $\mu$ minus one.
\end{thm}

As an example, the starred boxes below form a skew $3$ hook in a tableau for $\symm{7}$:
\begin{equation*}
\young(\;\;*,\;**,\;)
\end{equation*}
in this case we have $r(\lambda,\mu)=1$.

Also note that it follows from Theorem \ref{thm_MN} that if we cannot remove a skew $m$ hook from $\lambda$ (i.e. there is no tableau $\mu$ that can be obtained by removing such a skew hook) and $g\in\symm{N}$ contains an $m$ cycle then:
\begin{equation*}
\chi^\lambda (g) =0
\end{equation*}

From the two theorems above, one can derive that (Theorem 4.56 of \cite{dBR}):
\begin{thm}\label{thm_char} If $a\in\symm{M}$ and $b$ is a product of $k$ cycles, each of length $m$, that leaves $\{1,\ldots,a \}$ fixed pointwise. And $\lambda$ a partition of $N=M+km$ out of which exactly $k$ skew $m$ hooks are removable then:
$$
\chi^\lambda (ab) = \sigma f_m^\lambda \chi^{\tilde{\lambda}} (a)
$$
where $\tilde{\lambda}$ is what is left over of $\lambda$ after the removal of $k$ skew $m$ hooks and is independent of the order of removal. Furthermore:
$$
\sigma = (-1)^{\sum\limits_{i=1}^k r(\mu_{i-1},\mu_i)}
$$
where $\mu_0=\lambda$, $\mu_k=\tilde{\lambda}$ and $\mu_i$ is a tableau that is obtainable from $\mu_{i-1}$ by the removal of a skew $m$ hook. Finally $f_m^\lambda$ is the number of ways to consecutively remove $k$ skew $m$ hooks from $\lambda$.
\end{thm} 

We will be interested in the case where there might be more skew $m$ hooks removable from a tableau $\lambda$ then there are $m$ cycles in the element $g\in\symm{N}$ (note again that if there are fewer skew $m$ hooks removable from $\lambda$ than $m$ cycles in $g$ then $\chi^\lambda(g)=0$). So, we suppose that $g$ contains $k$ $m$ cycles. We write $g=ab$ where $b$ a product of $k$ $m$ cylces and $a$ contains no such cycle. Then we have:
\begin{equation*}
\chi^\lambda(g) = \sum_\mu \sigma_\mu \chi^\mu (a)
\end{equation*}
Where the sum is over diagrams $\mu$ that can be obtained from $\lambda$ by removing $k$ skew $m$ hooks and $\sigma_\mu$ is the power of $-1$ that comes out of Theorem \ref{thm_MN}. This means that:
\begin{equation*}
\abs{\chi^\lambda(g)} \leq  \max\verz{\abs{\chi^\mu (a)}}{a\in\symm{N-km},\;\mu\text{ a partition of }N-km} f_{k,m}^\lambda
\end{equation*}
where $f_{k,m}^\lambda$ is the number of ways to remove $k$ skew $m$ hooks from $\lambda$. We have: 
\begin{equation*}
f_{k,m}^\lambda \leq f_{m}^\lambda
\end{equation*}
and hence:
\begin{equation*}
\abs{\chi^\lambda(g)} \leq  \max\verz{\abs{\chi^\mu (a)}}{a\in\symm{N-km},\;\mu\text{ a partition of }N-km} f_{m}^\lambda
\end{equation*}

The last fact about the characters of $\symm{N}$ we need is the following theorem by Fomin and Lulov:
\begin{thm}\cite{FL}\label{thm_f} Let $\lambda$ be a partition of $N$. Then:
Let $\lambda\models N=M+km$ such that exactly $k$ skew $m$ hooks can be removed from $\lambda$. Then:
$$
f^\lambda_m \leq \frac{k! \; m^k}{(N!)^{1/m}} \left(f^\lambda\right)^{1/m}
$$
\end{thm}
In fact, Fomin and Lulov state the theorem only in the case $M=0$, but their proof works verbatim in this slightly more general case.

\subsection{Representations of the alternating group} We start with the conjugacy classes of $\alt{N}$ (Lemma 1.2.10 of \cite{JK}):
\begin{lem}
Let $\lambda$ be a partition of $N$. Then:
\begin{itemize}
\item if $\lambda$ contains an odd number of even parts $K(\lambda)\cap\alt{N}=\emptyset$
\item if the parts of $\lambda = (\lambda_1,\ldots,\lambda_k)$ are all pairwise different and odd then $K(\lambda)\cap\alt{N}$ splits into two conjugacy classes $K(\lambda)^+$ and $K(\lambda)^-$ of equal size. By convention we take: 
  $$(1\ldots \lambda_1)(\lambda_1+1\ldots \lambda_1+\lambda_2)\cdots(N-\lambda_k+1\ldots N)\in K(\lambda)^+$$
\item otherwise $K(\lambda)\cap\alt{N}$ is a conjugacy class of $\alt{N}$.
\end{itemize}
\end{lem}

We need to define the notion of the \emph{associated partition} of a partition $\lambda$. This is the partition $\lambda'$ obtained by reflecting $\lambda$ in its main diagonal. Furthermore, if a vector space $V$ is a representation of $\symm{N}$ then the restriction of $V$ to $\alt{N}$ will be denoted $V\downarrow_{\alt{N}}$. We have the following theorem (Theorem 2.5.7 of \cite{JK}):
\begin{thm} 
Suppose $\lambda$ is a partition of $N$ then:
\begin{itemize}
\item If $\lambda \neq \lambda'$ then $V^\lambda\downarrow_{\alt{N}}=V^{\lambda'}\downarrow_{\alt{N}}$ is an irreducible representation of $\alt{N}$.
\item If $\lambda = \lambda'$ then $V^\lambda\downarrow_{\alt{N}}=V^{\lambda'}\downarrow_{\alt{N}}$ splits into two unequivalent irreducible representations $V^\lambda_+$ and $V^\lambda_-$.
\end{itemize}
\end{thm}

It follows from counting conjugacy classes that the theorem above gives us a complete list of irreducible representations.

Now we want to express the characters associated to these representations in terms of the characters of $\symm{N}$. It follows immediately from the theorem above that the characters for non self-associated partitions are identical. To avoid confusion we will however denote the $\alt{N}$-character corresponding to $\lambda$ by $\zeta^\lambda$. For self-associated partitions we have the following theorem (Theorem 2.5.13 of \cite{JK}):

\begin{thm}\label{thm_sachar} If $\lambda$ is a partition of $N$ such that $\lambda = \lambda'$ and the main diagonal of $\lambda$ has length $k$. We write:
$$
H^+(\lambda) = K(h(1,1),\ldots, h(k,k))^+
$$
$$
H^-(\lambda) = K(h(1,1),\ldots, h(k,k))^-
$$
Then for $\pi\in\alt{N}$ the $\alt{N}$-characters corresponding to $\lambda$ are given by:
$$
\zeta^\lambda_\pm (\pi) = \left\{
\begin{array}{ll}
\frac{1}{2}\left( (-1)^{(N-k)/2} \pm \sqrt{(-1)^{(N-k)/2} \prod\limits_{i=1}^k h(i,i)} \right) & \text{if } \pi\in H^+(\lambda) \\[5mm]
\frac{1}{2}\left( (-1)^{(N-k)/2} \mp \sqrt{(-1)^{(N-k)/2} \prod\limits_{i=1}^k h(i,i)} \right) & \text{if } \pi\in H^-(\lambda) \\[5mm]
\frac{1}{2}\chi^\lambda (\pi) & \text{otherwise}
\end{array}
\right.
$$
where $\chi^\lambda (\pi)$ is the character of $\pi$ as an element of $\symm{N}$.
\end{thm}

Finally, we have the following lemma about the values of the $\symm{N}$-characters of self associated partitions (Lemma 2.5.12 of \cite{JK}):
\begin{lem}\label{lem_sasn}
If $\lambda$ is a partition of $N$ such that $\lambda = \lambda'$ and the main diagonal of $\lambda$ has length $k$. Then:
$$
\chi^\lambda\left(K(h(1,1),\ldots, h(k,k))\right) = (-1)^{(N-k)/2}
$$
\end{lem}
\subsection{Some inequalities for self associated tableaux}

We will be relating the characters of the alternating group to those of the symmetric group. Theorem \ref{thm_sachar} shows us that the characters corresponding to self associated partitions might cause a problem. To solve this, we have the following upper bounds:

\begin{prp}\label{prp_ubsa1} There exists a constant $A>0$ independent of $N$ such that for any partition $\lambda$ of $N$ with $\lambda'=\lambda$ we have:
$$
\frac{\prod\limits_{i=1}^d h(i,i)}{f^\lambda} \leq A^N N^{\sqrt{N}-N/2}
$$
where $d$ is the number of boxes in the main diagonal of $\lambda$.
\end{prp} 

\begin{proof}
The hook length formula (Theorem \ref{thm_hook}) gives us:
\begin{equation*}
\frac{\prod\limits_{i=1}^d h(i,i)}{f^\lambda}=\frac{\prod\limits_{i=1}^d h(i,i) \prod\limits_{i,j=1}^r h(i,j)}{N!}
\end{equation*}
where $r$ is the number of rows in $\lambda$. Hence, by the arithmetic geometric mean inequality we get:
\begin{align*}
\frac{\prod\limits_{i=1}^d h(i,i)}{f^\lambda} & \leq \frac{\left(\frac{1}{N+d}\left(\sum\limits_{i=1}^d h(i,i)+\sum\limits_{i,j=1}^r h(i,j) \right)\right)^{N+d}}{N!} \\
 & = \frac{\left(\frac{2}{N+d}\left(\sum\limits_{i\leq j=1}^r h(i,j) \right)\right)^{N+d}}{N!} 
\end{align*}
where the last step follows from the fact that $\lambda = \lambda'$ and hence that $h(i,j)=h(j,i)$ for all $i,j=1,\ldots,r$. From the fact that $\lambda'=\lambda$ we also get that:
\begin{equation*}
h(i,j) \leq \frac{1}{2} \left(h(i,i)+h(j,j) \right)
\end{equation*}
for $i,j=1,\ldots, r$, where we set $h(i,i)=0$ if $(i,i)\notin \lambda$. Note that we have equality if and only if both $(i,i)\in\lambda$ and $(j,j)\in\lambda$. So:
\begin{align*}
\frac{\prod\limits_{i=1}^d h(i,i)}{f^\lambda} & \leq  \frac{\left(\frac{1}{N+d}\left(\sum\limits_{i\leq j=1}^r h(i,i)+h(j,j) \right)\right)^{N+d}}{N!}  \\
   & =  \frac{\left(\frac{1}{N+d}\left(\sum\limits_{i=1}^d (d-i+1)h(i,i)+\sum\limits_{i=1}^d i h(j,j) \right)\right)^{N+d}}{N!} \\
   & = \frac{\left(\frac{(d+1)N}{N+d} \right)^{N+d}}{N!} \\
   & \leq \frac{(d+1)^{N+d}}{N!} \\
\end{align*}
We have $d\leq \sqrt{N}$, because a tableau with a main diagonal of $d$ boxes must contain a $d\times d$ square. Hence:
\begin{align*}
\frac{\prod\limits_{i=1}^d h(i,i)}{f^\lambda} & \leq \frac{(\sqrt{N}+1)^{N+\sqrt{N}}}{N!} \\
 & \leq A^N\frac{\sqrt{N}^{N+\sqrt{N}}}{N^N} \\
  & \leq A^N\frac{1}{N^{N/2-\sqrt{N}}} \\
\end{align*}
for some constant $A>0$ independent of $N$, which comes out of Stirling's approximation. Note that we could get explicit constants $A$ and $B$, but since they won't be needed and will only complicate the formulas, we choose not to compute them.
\end{proof}

\begin{prp}\label{prp_ubsa2}
For any partition $\lambda$ of $N$ with $\lambda'=\lambda$ we have:
$$
\frac{1}{f^\lambda} \leq  \frac{\left(\sqrt{N}+1\right)^{N}}{N!} 
$$
\end{prp} 

\begin{proof}
From the arithmetic-geometric mean inequality we obtain:
\begin{align*}
\frac{1}{f^\lambda} & = \frac{\prod\limits_{(i,j)\in\lambda}h(i,j)}{N!} \\
 & \leq \frac{\left(\frac{1}{N} \sum\limits_{(i,j)\in\lambda}h(i,j)\right)^{N}}{N!} \\
 & = \frac{\left(\frac{1}{N} \left(2\sum\limits_{i<j,(i,j)\in\lambda}h(i,j)+\sum\limits_{i=1}^d h(i,i) \right)\right)^{N}}{N!} \\
\end{align*}
Reasoning in a similar way to the previous proof we get: 
\begin{align*} 
 \frac{1}{f^\lambda} & \leq \frac{\left(\frac{1}{N} \left(dN+N \right)\right)^{N}}{N!}  \\
 & = \frac{\left(d+1\right)^{N}}{N!} \\
 & \leq \frac{\left(\sqrt{N}+1\right)^{N}}{N!}
\end{align*}
\end{proof}

\subsection{Other bounds}

Finally, we need two more bounds. The first one is an upper bound on the number of partitions of a number $N\in\mathbb{N}$. To us it will only matter that this upper bound is subexponential in $N$. For reference, we include the following theorem (that can be found as Theorem 14.5 in \cite{Apo}):
\begin{thm}
Let $p(N)$ be the number of partitions of the number $N\in\mathbb{N}$. Then:
$$
p(N) < \exp\left(\pi\sqrt{\frac{2N}{3}}\right)
$$
\end{thm}

The second one is the following proposition (which appears as Theorem 1.1 in \cite{LS} and Proposition 4.2 in \cite{Gam}):

\begin{prp}\label{prp_gam}\cite{LS}\cite{Gam}
For any $t>0$ and $m\in\mathbb{N}$ we have:
$$
\sum\limits_{\substack{\lambda\models N \\ \lambda \neq (6N-2M),(1,1,\ldots,1) \\ \lambda_1,\lambda_1'\leq N-m}} \left(f^\lambda \right)^{-t} = \mathcal{O}\left(N^{-mt} \right)
$$
\end{prp}

\subsection{The proof}

Now we can prove Theorem \ref{thm_limit}, which we repeat for the reader's convenience:
\begin{thmc}{\ref{thm_limit}}
Let $W$ be a finite set of words and $m\in \mathbb{N}^W$ then:
$$
\lim\limits_{N\rightarrow\infty}\norm{\mathbb{P}_{3\star 2,N,W,m} - \mathbb{U}_{N,W,m}} = 0
$$
where $\mathbb{U}_{N,W,m}$ denotes the uniform probability measure on $\alt{6N-2M}$.
\end{thmc}

\begin{proof}
To lighten notation we are going to drop the subscripts in $\mathbb{P}_{3\star 2,N,W,m}$ and $\mathbb{U}_{N,W,m}$ and we will write $r=6N-2M$. Furthermore, characters denoted with a $\zeta$ will always be $\alt{N}$-characters and characters denoted with a $\chi$ will always be $\symm{N}$-characters. 

The Diaconis-Shahshahani upper bound lemma (Lemma \ref{lem_DS}) in combination with Lemma \ref{lem_alt} that tells us that the product $\sigma\tau$ lies in $\alt{r}$ gives us:
\begin{align*}
\norm{\mathbb{P} - \mathbb{U}}^2 & \leq \frac{1}{4}\sum\limits_{\substack{\rho\in \widehat{\alt{r}} \\ \rho\neq \mathrm{id}}} \dim(\rho) \tr{\hat{\mathbb{P}}(\rho) \overline{\hat{\mathbb{P}}(\rho)}} \\
   & = \frac{1}{4}\sum\limits_{\substack{\rho\in \widehat{\alt{r}} \\ \rho\neq \mathrm{id}}} \frac{1}{\dim(\rho)} \sum\limits_{\substack{K,L\text{ conjugacy} \\ \text{classes of }\alt{r}}}\Pro{}{K}\Pro{}{L}\aant{K}\aant{L} \zeta^\rho(K)\zeta^\rho(L)
\end{align*}
Where we have used the fact that:
\begin{equation*}
\hat{\mathbb{P}}(\rho) = \frac{1}{\dim(\rho)}\sum\limits_{K\text{ conjugacy class of }\alt{r}} \Pro{}{K} \aant{K}  \zeta^\rho(K) I_{\dim(\rho)}
\end{equation*}
where $\Pro{}{K}=\Pro{}{\pi}$ for any $\pi\in K$ (and is not to be confused with the probability of obtaining an element in $K$, which is equal to $\Pro{}{K}\aant{K}$). This follows from the fact that $\mathbb{P}$ is constant on conjugacy classes and Schur's lemma (see also Lemma 5 of \cite{DS}). 

If $K$ is a conjugacy class of $\symm{r}$ such that $K\cap \alt{r}=\emptyset$ then it follows from Lemma \ref{lem_alt} that $\Pro{}{K} = 0$. This means that we can add all these conjugacy classes to the sum above. Furthermore, Theorem \ref{thm_sachar} tells us how to relate $\alt{r}$-characters to $\symm{r}$-characters, so we get:
\begin{align*}
\norm{\mathbb{P} - \mathbb{U}}^2 & \leq & \frac{1}{2} \sum\limits_{\substack{\lambda \models r, \lambda\neq\lambda',\\ \lambda \neq (r), (1,1,\ldots,1)}} \sum\limits_{\substack{K,L\text{ conjugacy} \\ \text{classes of }\symm{r}}} \frac{\Pro{}{K}\Pro{}{L}\aant{K}\aant{L} \chi^\lambda(K)\chi^\lambda(L)}{f^\lambda} \\
&    & + \sum\limits_{\substack{\lambda \models r \\ \lambda =\lambda'}}  \sum\limits_{\substack{L\text{ conjugacy} \\ \text{class of }\symm{r} \\ L\neq K(\lambda)}} \frac{\Pro{}{H^+(\lambda)}\Pro{}{L}\aant{H^+(\lambda)}\aant{L} \zeta^\lambda(H^+(\lambda))\chi^\lambda(L)}{f^\lambda} \\
&    & + \sum\limits_{\substack{\lambda \models r \\ \lambda =\lambda'}}  \sum\limits_{\substack{L\text{ conjugacy} \\ \text{class of }\symm{r} \\ L\neq K(\lambda)}} \frac{\Pro{}{H^-(\lambda)}\Pro{}{L}\aant{H^-(\lambda)}\aant{L} \zeta^\lambda(H^-(\lambda))\chi^\lambda(L)}{f^\lambda} \\
&    & + \sum\limits_{\substack{\lambda \models r \\ \lambda =\lambda'}}  \sum_{i,j = \pm}\frac{\Pro{}{H^i(\lambda)}\aant{H^i(\lambda)}\Pro{}{H^j(\lambda)}\aant{H^j(\lambda)} \zeta^\lambda(H^i(\lambda))\zeta^\lambda(H^j(\lambda))}{f^\lambda}
\end{align*}
We now use the fact that value of a $\symm{r}$ character of a self associated partition $\lambda$ on $H^\pm(\lambda)$ is a power of $-1$ (Lemma \ref{lem_sasn}) to obtain:
\begin{align*}
\norm{\mathbb{P} - \mathbb{U}}^2 & \leq  & \frac{1}{2} \sum\limits_{\substack{\lambda \models r, \lambda\neq\lambda' \\ \lambda \neq (r),(1,1,\ldots,1)}}  f^\lambda \tr{\widehat{\mathbb{P}(\lambda)} \overline{\widehat{\mathbb{P}(\lambda)}} } \\
&    & + \sum\limits_{\substack{\lambda \models r \\ \lambda =\lambda'}} \Pro{}{H^+(\lambda)}\aant{H^+(\lambda)} \zeta^\lambda(H^+(\lambda))\left(\tr{\widehat{\mathbb{P}(\lambda)}} +\frac{2}{f^\lambda} \right)\\
&    & + \sum\limits_{\substack{\lambda \models r \\ \lambda =\lambda'}} \Pro{}{H^-(\lambda)}\aant{H^-(\lambda)} \zeta^\lambda(H^-(\lambda))\left(\tr{\widehat{\mathbb{P}(\lambda)}}+\frac{2}{f^\lambda} \right) \\
&    & + \sum\limits_{\substack{\lambda \models r \\ \lambda =\lambda'}}  \sum_{i,j = \pm}\frac{\Pro{}{H^i(\lambda)}\aant{H^i(\lambda)}\Pro{}{H^j(\lambda)}\aant{H^j(\lambda)} \zeta^\lambda(H^i(\lambda))\zeta^\lambda(H^j(\lambda))}{f^\lambda} \\
\end{align*}
We want to get rid of the last three sums, because the first sum is the analogue of the one that appears in the proof by Gamburd \cite{Gam}. For this we are going to use Theorem \ref{thm_sachar}, Propositions \ref{prp_ubsa1} and \ref{prp_ubsa2} and estimates similar to the ones in the proofs of these propositions. For a self associated partition $\lambda$ with $d$ blocks on its main diagonal we have:
\begin{align*}
\abs{\Pro{}{H^\pm(\lambda)}\aant{H^\pm(\lambda)} \zeta^\lambda(H^\pm(\lambda))} & \leq \abs{\zeta^\lambda(H^+(\lambda))} \\
   & \leq 1+\prod\limits_{i=1}^d h(i,i)
\end{align*}
Using the arithmetic geometric mean inequality we get:
\begin{align*}
\abs{\Pro{}{H^\pm(\lambda)}\aant{H^\pm(\lambda)} \zeta^\lambda(H^\pm(\lambda))} & \leq 1+\left(\frac{1}{d}\sum\limits_{i=1}^d h(i,i)\right)^d \\
   & = 1+\left(\frac{N}{d}\right)^d
\end{align*}
Furthermore, because now we are working in the symmetric group and the Fourier transform turns convolution into ordinary multiplication (see for instance Lemma 1 of \cite{DS}), we have:
\begin{align*}
\tr{\widehat{\mathbb{P}(\lambda)}} & = \tr{\widehat{\mathbb{P}_3(\lambda)}\widehat{\mathbb{P}_2(\lambda)}} \\
   &  = \frac{\chi^\lambda(K_3)\chi^\lambda(K_2)}{f^\lambda} 
\end{align*}

In Section \ref{sec_charsn} we have already seen that if an element $g\in\symm{N}$ contains $k$ cycles of length $m$ then:
\begin{equation*}
\abs{\chi^\lambda(g)} \leq  \max\verz{\abs{\chi^\mu (a)}}{a\in\symm{N-km},\;\mu\models N-km} f_{m}^\lambda
\end{equation*}
This means that:
\begin{equation*}
\abs{\chi^\lambda(K_3)} \leq  \max\verz{\abs{\chi^\mu (a)}}{a\in\symm{M},\;\mu\models M} f_{3}^\lambda
\end{equation*}
and:
\begin{equation*}
\abs{\chi^\lambda(K_2)} = f_{2}^\lambda
\end{equation*}
because $\tau$ contains only $2$-cycles. Note that the first factor in the upper bound for $\sigma$ does not depend on $N$ but only on the finite set of words $W$ we fix. We will write:
\begin{equation*}
\abs{\chi^\lambda(K_3)} \leq  C f_{3}^\lambda
\end{equation*}
So we obtain:
\begin{align*}
 \sum\limits_{\substack{\lambda \models r \\ \lambda =\lambda'}} \Pro{}{H^+(\lambda)}\aant{H^+(\lambda)} \zeta^\lambda(H^+(\lambda))\left(\tr{\widehat{\mathbb{P}(\lambda)}} +\frac{2}{f^\lambda} \right) \\
 \leq \sum\limits_{\substack{\lambda \models r \\ \lambda =\lambda'}} \Pro{}{H^+(\lambda)}\aant{H^+(\lambda)} \zeta^\lambda(H^+(\lambda))\frac{Cf_3^\lambda f_2^\lambda +2}{f^\lambda} 
\end{align*}
Now we are going to use the upper bound on $f_2^\lambda$ and $f_3^\lambda$ of Theorem \ref{thm_f}, which gives us:
\begin{equation*}
\frac{Cf_3^\lambda f_2^\lambda +2}{f^\lambda} \leq \frac{C\frac{k_3! \; 3^{k_3}}{(r!)^{1/3}} \left(f^\lambda\right)^{1/3} \frac{k_2! \; 2^{k_2}}{(r!)^{1/2}} \left(f^\lambda\right)^{1/2}+2}{f^\lambda}
\end{equation*}
Where $k_2$ and $k_3$ are the numbers of skew $2$ and $3$ hooks that can be removed from $\lambda$. We have: $k_2\leq r/2$ and $k_3\leq r/3$. Hence:
\begin{align*}
\frac{Cf_3^\lambda f_2^\lambda +2}{f^\lambda} & \leq \frac{C (r/3)! \; 3^{r/3} \; (r/2)! \; 2^{r/2}}{(r!)^{5/6}\left(f^\lambda\right)^{1/6}}+\frac{2}{f^\lambda} \\
& \leq C'\frac{\sqrt{r}\sqrt{r}2^{r/2}3^{r/3} \left(\frac{r}{2e} \right)^{r/2} \left(\frac{r}{3e} \right)^{r/3}}{r^{5/12}\left(\frac{r}{e} \right)^{5r/6}}\frac{1}{\left(f^\lambda\right)^{1/6}} +\frac{2}{f^\lambda} \\
& = C' r^{7/12}\frac{1}{\left(f^\lambda\right)^{1/6}} + \frac{2}{f^\lambda}
\end{align*}
where the second inequality comes from Stirling's approximation. We will now use the upper bound for $\frac{1}{f^\lambda}$ for $\lambda$ self associated from Proposition \ref{prp_ubsa2}. Combining this with all the above, we get:
\begin{align*}
 \sum\limits_{\substack{\lambda \models r \\ \lambda =\lambda'}} \Pro{}{H^\pm(\lambda)}\aant{H^\pm(\lambda)} \zeta^\lambda(H^\pm(\lambda))\left(\tr{\widehat{\mathbb{P}(\lambda)}} +\frac{2}{f^\lambda} \right) \\
  \leq \sum\limits_{\substack{\lambda \models r \\ \lambda =\lambda'}} \left(1+\left(\frac{r}{d_\lambda}\right)^{d_\lambda}\right) \left(C'r^{7/2} \left(\frac{\left(\sqrt{r}+1\right)^{r}}{r!}\right)^{1/6}+2\frac{\left(\sqrt{r}+1\right)^{r}}{r!}   \right) \\[3mm]
 \leq p(r) \left(1+\left(r\right)^{\sqrt{r}}\right) \left(C'r^{7/2} \left(\frac{\left(\sqrt{r}+1\right)^{r}}{r!}\right)^{1/6} +2\frac{\left(\sqrt{r}+1\right)^{r}}{r!}   \right) \\
 \leq C''A^{r} r^{\sqrt{r}}\left(\frac{1}{r^{r/12}}  +2\frac{1}{r^{r/2}}   \right) \\
  \leq C'''A^{r} r^{\sqrt{r}-r/12}  
\end{align*}
for constants $A,C'',C'''>0$ independent of $r$. For $r\rightarrow\infty$ this tends to $0$ . For the final term of the sum above we need Proposition \ref{prp_ubsa1} and Theorem \ref{thm_sachar}. We have:
\begin{align*}
\frac{\abs{\Pro{}{H^i(\lambda)}\aant{H^i(\lambda)}\Pro{}{H^j(\lambda)}\aant{H^j(\lambda)} \zeta^\lambda(H^i(\lambda))\zeta^\lambda(H^j(\lambda))}}{f^\lambda} & \leq  \frac{\abs{\zeta^\lambda(H^i(\lambda))\zeta^\lambda(H^j(\lambda))}}{f^\lambda} \\
 & \leq C'\frac{\prod\limits_{i=1}^d h(i,i)}{f^\lambda}
\end{align*}
for some $C'>0$ independent of $r$, where we have used Theorem \ref{thm_sachar} for the final step. Now we apply Proposition \ref{prp_ubsa1} and we get:
\begin{align*}
\frac{\abs{\Pro{}{H^i(\lambda)}\aant{H^i(\lambda)}\Pro{}{H^j(\lambda)}\aant{H^j(\lambda)} \zeta^\lambda(H^i(\lambda))\zeta^\lambda(H^j(\lambda))}}{f^\lambda} & \leq  C' 
A^{r} r^{\sqrt{r}-r/2}
\end{align*}
for some $A>0$ independent of $r$. So the only term in the Diaconis-Shahshahani upper bound we are concerned with now is:
\begin{align*}
\frac{1}{2} \sum\limits_{\substack{\lambda \models r, \lambda\neq\lambda' \\ \lambda \neq (r),(1,1,\ldots,1)}}  f^\lambda \tr{\widehat{\mathbb{P}(\lambda)}\overline{\widehat{\mathbb{P}(\lambda)}} } & = \frac{1}{2} \sum\limits_{\substack{\lambda \models r, \lambda\neq\lambda' \\ \lambda \neq (r),(1,1,\ldots,1)}}  \left(\frac{\chi^\lambda(K_3)\chi^\lambda(K_2)}{f^\lambda} \right)^2 \\
\end{align*}
We have:
\begin{align*}
 \left(\frac{\chi^\lambda(K_3)\chi^\lambda(K_2)}{f^\lambda} \right)^2 & \leq C^2 \left(\frac{f_3^\lambda f_2^\lambda}{f^\lambda} \right)^2 
\end{align*}
Now we use Theorem \ref{thm_f} again in combination with the fact that at most $r/m$ skew $m$ hooks can be removed from a tableau of $r$ boxes to obtain:
\begin{align*}
 \left(\frac{\chi^\lambda(K_3)\chi^\lambda(K_2)}{f^\lambda} \right)^2 & \leq C^2  \left(\frac{\frac{(r/3)! \; 3^{r/3}}{(r!)^{1/3}} \; \frac{(r/2)! \; 2^{r/2}}{(r!)^{1/2}} }{\left(f^\lambda\right)^{1/6}} \right)^2 \\
& \leq \frac{B\cdot C^2 2^{1/12}}{3^{1/2}} (\pi r)^{7/12} \frac{1}{\left(f^\lambda\right)^{1/3}}
\end{align*}
for some $B\in (0,\infty)$ coming from Stirling's approximation. Finally we apply Proposition \ref{prp_gam} which tells us that: 
\begin{equation*}
\sum\limits_{\substack{\lambda \models r \\ \lambda \neq (r),(1,1,\ldots,1)\\ \lambda_1,\lambda_1'\leq r-4}}  \frac{1}{\left(f^\lambda\right)^{1/3}} = \mathcal{O}\left(r^{-\frac{4}{3}} \right)
\end{equation*}
To estimate the remaining terms we need to make a similar table to Table 1 in \cite{Gam}. It turns out that the only partition that gives us a problem is the partition $(r-1,1)$. All the other partitions in Table 1 of \cite{Gam} have dimensions quadratic in $r$ and hence add a term $r^{-2/3}$ in total.

We have:
\begin{equation*}
f^{(r-1,1)} = r-1
\end{equation*}
and a straight forward application of the Murnaghan-Nakayama rule (Theorem \ref{thm_MN}) gives us:
\begin{equation*}
\abs{\chi^{(r-1,1)}(K_3)}  \leq n_1+1
\end{equation*}
where $n_1$ is the number of singleton cycles in $\sigma$, which is a constant in our considerations (because it only depends on $W$ and $m$). And:
\begin{equation*}
\abs{\chi^{(r-1,1)}(K_2)} = 1
\end{equation*}
Hence these partitions add a term $\frac{n_1+1}{r-1}$. Note that all the terms we found limit to $0$ as $r\rightarrow\infty$.

Summing all the estimates above concludes the proof.
\end{proof}

\subsection{The consequences of Theorem \ref{thm_limit}} In this section we study the consequences of Theorem \ref{thm_limit}. The first one is the following:

\begin{cor}\label{cor_prob} Let $W$ be a finite set of words in $L$ and $R$ not containing any words of the form $[L^n]$ for some $n\in\mathbb{N}$, $m\in\mathbb{N}^W$ and $\Gamma_{(W,m)}$ a labelled set of circuits representing $(W,m)$. Furthemore, let the sequence of subsets $D_N\subset\mathbb{N}$ be non-negligible with respect to the genus. Then:
$$
\lim\limits_{N\rightarrow\infty}\frac{\ProC{N}{\Gamma_{(W,m)}\subset\Gamma}{g\in D_N}}{\Pro{N}{\Gamma_{(W,m)}\subset\Gamma}} = 1
$$
\end{cor}

\begin{proof} We have:
\begin{align*}
\ProC{N}{\Gamma_{(W,m)}\subset\Gamma}{ g\in D_N} & = \frac{\Pro{N}{\Gamma_{(W,m)}\subset\Gamma \text{ and } g\in D_N }}{\Pro{N}{ g\in D_N}} \\
   & = \frac{\ProC{N}{g\in D_N }{ \Gamma_{(W,m)}\subset\Gamma}}{\Pro{N}{ g\in D_N }}\Pro{N}{\Gamma_{(W,m)}\subset\Gamma} \\
\end{align*}
From Theorem \ref{thm_limit} in combination Gamburd's Theorem \cite{Gam} we know that:
\begin{equation*}
\abs{\ProC{N}{g\in D_N }{ \Gamma_{(W,m)}\subset\Gamma} - \Pro{N}{ g\in D_N }} \rightarrow 0
\end{equation*}
for $N\rightarrow\infty$. Furthermore, because we have assumed that the sequence $D_N\subset\mathbb{N}$ is non-negligble, we have:
\begin{equation*}
\liminf_{N\rightarrow\infty} \Pro{N}{g\in D_N} >0
\end{equation*}
So:
\begin{equation*}
\frac{\ProC{N}{\Gamma_{(W,m)}\subset\Gamma }{ g\in D_N }}{\Pro{N}{\Gamma_{(W,m)}\subset\Gamma}} \rightarrow 1
\end{equation*}
as $N\rightarrow\infty$.
\end{proof}

The reason we are interested in the above corollary, is the following consequence:

\begin{cor}\label{cor_JFM} Let $W$ be a finite set of words in $L$ and $R$ not containing any words of the form $[L^n]$ for some $n\in\mathbb{N}$, $m\in\mathbb{N}^W$. Furthemore, let the sequence of subsets $D_N\subset\mathbb{N}$ be non-negligible wirth respect to the genus. Then:
$$
\lim\limits_{N\rightarrow\infty} \ExV{N}{\prod\limits_{w\in W}\left(Z_{[w]} \right)_{m_w}\; | \; g\in D_N} = \lim\limits_{N\rightarrow\infty} \ExV{N}{\prod\limits_{w\in W}\left(Z_{[w]} \right)_{m_w}}
$$
\end{cor}

\begin{proof} For any choice of labelled graph $\Gamma_{(W,m)}$ representing $(W,m)$ as disjoint circuits we have:
\begin{equation*}
 \ExV{N}{\prod\limits_{w\in W}\left(Z_{[w]} \right)_{m_w}\; | \; g\in D_N } = a_{N,(W,m)} \cdot \Pro{N}{\Gamma_{(W,m)}\subset\Gamma \; |\; g\in D_N}
\end{equation*}
where $a_{N,(W,m)}$ counts the number of ways of realizing $(W,m)$ as a graph $\Gamma_{(W,m)}$. We note that the number $a_{N,(W,m)}$ is independent of the restrictions on the genus of a random surface. That is, we have:
\begin{equation*}
\ExV{N}{\prod\limits_{w\in W}\left(Z_{[w]} \right)_{m_w}} = a_{N,(W,m)} \cdot \Pro{N}{\Gamma_{(W,m)}\subset\Gamma }
\end{equation*}

So we get:
\begin{align*}
\lim\limits_{N\rightarrow\infty} \ExV{N}{\prod\limits_{w\in W}\left(Z_{[w]} \right)_{m_w}\; | \; g\in D_N}  & = \lim\limits_{N\rightarrow\infty} a_{N,(W,m)} \cdot \Pro{N}{\Gamma_{(W,m)}\subset\Gamma \; |\; g\in D_N} \\
  & = \lim\limits_{N\rightarrow\infty} \ExV{N}{\prod\limits_{w\in W}\left(Z_{[w]} \right)_{m_w}} \cdot \frac{\Pro{N}{\Gamma_{(W,m)}\subset\Gamma \; |\; g\in D_N}}{\Pro{N}{\Gamma_{(W,m)}\subset\Gamma }} \\
  & = \lim\limits_{N\rightarrow\infty} \ExV{N}{\prod\limits_{w\in W}\left(Z_{[w]} \right)_{m_w}}
\end{align*}
In the above we have cheated slightly, we have skipped over realizations of $(W,m)$ as intersecting circuits, but asymptotically these do not contribute anything. In fact, it is part of the proof of Bollob\'as' theorem on circuits in random regular graphs that in the unrestricted case this contribution is $\mathcal{O}(N^{-1})$. Because our condition on the genus is non-negligible this gives us an upper bound for the contribution of intersecting circuits in the restricted case as well.
\end{proof}

This in turn implies Theorem A:

\begin{thmA} Let $D_N\subset\mathbb{N}$ for all $N\in\mathbb{N}$ be non-negligible with respect to the genus. Furthermore, let $W$ be a finite set of equivalence classes of words not containing $[L^n]$ for any $n$. Then we have: 
$$
\left(Z_{N,[w]}\right)\vert_{g\in D_N}\rightarrow Z_{[w]} \text{ in distribution for }N\rightarrow\infty
$$
for all $[w]\in W$, where the limit has to be taken over all even $N$ and where:
\begin{itemize}[leftmargin=0.2in]
\vspace{-0.1in}
\item $Z_{[w]}:\mathbb{N}\rightarrow\mathbb{N}$ is a Poisson distributed random variable with mean $\lambda_{[w]}=\frac{\aant{[w]}}{2\abs{w}}$ for all $w\in W$.
\item The random variables $Z_{[w]}$ and $Z_{[w']}$ are independent for all $[w],[w']\in W$ with $[w]\neq [w']$.
\end{itemize}
\end{thmA}

\begin{proof}
Using the method of moments (see for instance Theorem 1.21 in \cite{Bol2}), this follows directly from the convergence of the joint factorial moments (Corollary \ref{cor_JFM}).
\end{proof}

The final question that needs to be answered in this setting is what it means for a condition on the genus to be non-negligible. From Gamburd's theorem and the properties of the Poisson-Dirichlet distribution  (see for instance page 58 of \cite{Pit} and Corollary 5.1 of \cite{Gam}), we get that:
\begin{equation*}
g \stackrel{d}{\sim} 1+\frac{N}{2}-\mathcal{N}(\log(2N),\sqrt{\log(2N)})
\end{equation*}
for $N\rightarrow \infty$, where $\mathcal{N}(\mu,\sigma)$ denotes the normal distribution with average $\mu$ and standard deviation $\sigma$. Furthermore, by $\stackrel{d}{\sim}$ we mean that the total variational distance between two distributions tends to $0$. This means that what we basically need for the condition to be non-negligble is that $D_N$ is some interval of a size comparable to the standard deviation and not too far away from the average. A precise example of this is that if there exist $C_1,C_2 \in (0,\infty)$ such that:
\begin{equation*}
\liminf_{N\rightarrow\infty} \frac{\aant{D_N}}{\sqrt{\log(2N)}} = C_1 \text{ and } \limsup_{N\rightarrow\infty} \frac{\max\limits_{x\in D_N}\{\abs{x-(\frac{N}{2}+1-\log(2N))}\}}{\sqrt{\log(2N)}} = C_2
\end{equation*}
then the sequence $\rij{D_N}{N\in\mathbb{N}}{}$ is non-negligible with respect to the genus.

\section{The proof of Theorem B}\label{sec_thmb}
What we would like to know is the behavior of the length spectrum for any restriction on the topology, i.e. without the non-negligibility condition. It is clear however that it is not possible to adapt the method of proof of Theorem A to this level of generality. Because that would imply that no matter what the restriction, the variables $Z_{N,[w]}$ converge to the Poisson-distributed random variables $Z_{[w]}$. This is not the case. For example, we could restrict to only triangulations of a disjoint union of $N$ puntured tori. On this set we have:
\begin{equation*}
Z_{N,[w]}(\omega) = \left\{
\begin{array}{ll}
3N & \text{if }[w]=[LR] \\
0 & \text{otherwise}
\end{array}
\right. \text{ for all }\omega\in\Omega_N
\end{equation*}
for all $N$, which clearly has a very different limit.

In this section we will focus on the case of maximal genus. We will consider the case of odd $N$, which means that the maximal genus we can attain is:
\begin{equation*}
g=\frac{N+1}{2}
\end{equation*}
In terms of the symmetric group description this means that $\sigma\tau$ contains a single cycle of full length.

We want to understand the distribution of the number of appearances of a number of words of a fixed type, under the condition that the genus of the surface is maximal. So, given a finite set of words $W$, not containing any left hand turn cycles, and $m\in \mathbb{N}^W$, we want to count the number of elements in the set:
\begin{equation*}
\verz{\omega\in\Omega_N}{\Gamma(\omega)\text{ has }m_w\text{ circuits carrying }w\text{ for all }w\in W\text{ and } g(\omega)=\frac{N+1}{2}}
\end{equation*}
Using the trick from Section \ref{sec_topgeom}, this is equivalent to counting the number:
\begin{equation*}
\aant{\verz{(\sigma,\tau) \in K_3(W,m)\times K_2(W,m)  }{\sigma\tau\text{ has }1\text{ cycle}}}
\end{equation*}
We have also already seen that in this number we count the same random surface many times, corresponding to the relabeling of vertices, or equivalently the choice of $\sigma$. This means that we can also fix a $\sigma\in K_3(W,m)$ and count the number:
\begin{equation*}
n(N,W,m) = \aant{\verz{\tau \in K_2(W,m)  }{\sigma\tau\text{ has }1\text{ cycle}}}
\end{equation*}
which is what we will do. We will use methods similar to those of Appendix 6 in \cite{BIZ}, where a similar number for gluings of quadrilaterals is counted.

For an element $\pi\in\symm{N}$ we denote the conjugacy class of $\pi$ by $K(\pi)$ and for two conjugacy classes $K,K'\subset\symm{N}$ we write:
\begin{equation*}
\delta_{K,K'}= \left\{
\begin{array}{ll}
1 & \text{if }K=K' \\
0 & \text{otherwise}
\end{array}
\right.
\end{equation*}
Furthermore, we set $K_3:=K_3(W,m)$ and $K_2:=K_2(W,m)$. Now:
\begin{equation*}
n(N,W,m) = \sum\limits_{\tau\in\symm{6N-2M}} \delta_{K(\tau),K_2}\delta_{K(\sigma\tau),K(6N-2M)}
\end{equation*}

We start with the following lemma.
\begin{lem} Let $W$ be a finite set of words in $L$ and $R$ and $m\in\mathbb{N}^W$. Then:
$$
n(N,W,m) =\frac{\aant{K_2}\cdot\aant{K(6N-2M)}}{(6N-2M)!}  \sum\limits_{p=0}^{6N-2M-1}\frac{(-1)^p}{f^p} \chi^p (K_2)\chi^p(K_3)
$$
\end{lem}

\begin{proof}
For any two elements $\alpha,\beta\in\symm{6N-2M}$ we have:
\begin{equation*}
\sum\limits_{\lambda\models 6N-2M} \chi^\lambda(\alpha)\chi^\lambda(\beta) = \frac{(6N-2M)!}{\aant{K(\alpha)}}\delta_{K(\alpha),K(\beta)}
\end{equation*}
So:
\begin{align*}
n(N,W,m) & =\frac{\aant{K_2}\cdot\aant{K(6N-2M)}}{((6N-2M)!)^2} \sum\limits_{\tau\in\symm{6N-2M}}\sum\limits_{\lambda,\mu\models 6N-2M}\chi^\lambda(\tau)\chi^\lambda(K_2)\chi^\mu(\sigma\tau)\chi^\mu(K(6N-2M)) \\
& =\frac{\aant{K_2}\cdot\aant{K(6N-2M)}}{((6N-2M)!)^2} \sum\limits_{\lambda,\mu\models 6N-2M}\chi^\lambda(K_2)\chi^\mu(K(6N-2M))\sum\limits_{\tau\in\symm{6N-2M}}\chi^\lambda(\tau)\chi^\mu(\sigma\tau)
\end{align*}
We have:
\begin{equation*}
\sum\limits_{\tau\in\symm{6N-2M}}\chi^\lambda(\tau)\chi^\mu(\sigma\tau) = \delta_{\lambda,\mu}\frac{(6N-2M)!}{f^\lambda}\chi^\lambda(\sigma)
\end{equation*}
This means that:
\begin{equation*}
n(N,W,m) =\frac{\aant{K_2}\cdot\aant{K(6N-2M)}}{(6N-2M)!}  \sum\limits_{\lambda\models 6N-2M}\frac{1}{f^\lambda} \chi^\lambda (K_2)\chi^\lambda(K(6N-2M))\chi^\lambda (\sigma)
\end{equation*}
The characters $\chi^\lambda(K(6N-2M))$ can be computed using Theorem \ref{thm_MN}. We have:
\begin{equation*}
\chi^\lambda(K(6N-2M)) = \left\{
\begin{array}{ll}
(-1)^p & \text{if }\lambda = (6N-2M-p,1^p) \\
0 & \text{otherwise}
\end{array}
\right.
\end{equation*}
Furthermore:
\begin{equation*}
\chi^\lambda (\sigma) = \chi^\lambda (K_3)
\end{equation*}
Because the sum above is now over all $\lambda\models 6N-2M$ of the form $(6N-2M-p,1^p)$, we will replace all indices $\lambda$ by indices $p$. So we get:
\begin{equation*}
n(N,W,m) =\frac{\aant{K_2}\cdot\aant{K(6N-2M)}}{(6N-2M)!}  \sum\limits_{p=0}^{6N-2M-1}\frac{(-1)^p}{f^p} \chi^p (K_2)\chi^p(K_3)
\end{equation*}
\end{proof}

We have:
\begin{equation*}
f^p = \binom{6N-2M-1}{p}
\end{equation*}
and:
\begin{equation*}
\chi^p(K_2) = (-1)^{\ceil{\frac{p}{2}}}\binom{3N-M-1}{\floor{\frac{p}{2}}}
\end{equation*}
So far we have adapted the computation of \cite{BIZ} to the trivalent case. For the next part of the computation we will need to use different methods.

To compute the characters $\chi^p(K_3)$ we will use the Murnaghan-Nakayama rule (Theorem \ref{thm_MN}). First we need some notation. We write: 
\begin{equation*}
K_M=K\left(\prod_{w\in W}l_w^{m_w}\cdot\prod_{w\in W}r_w^{m_w}\right) \subset\symm{M}
\end{equation*}

We have the following lemma:
\begin{lem}
Let $0\leq p\leq 6N-2M$. Then:
$$
\chi^p(K_3) = \sum_{\substack{0\leq r \leq \min\{M-1,p\} \\ 3|p-r}} \chi^{r}(K_M) \binom{2N-M}{\frac{p-r}{3}}
$$
where $\chi^{r}$ is the character of $\symm{M}$ corresponding to the partition $(M-r,1^r)$.
\end{lem}

\begin{proof} The idea is to remove skew $3$ hooks from every $\lambda_p=(6N-2M-p,1^p)$ until we arrive at a Young tableau for $\symm{M}$. This will allow us to express $\chi^p(K_3)$ in terms of the characters of $\symm{M}$. We have:
\begin{equation*}
\chi^{(3)}(K(3)) = \chi^{(1,1,1)}(K(3)) = 1
\end{equation*}
and these are the only skew $3$ hooks we can remove from $\lambda_p$. There are $2N-M$ skew $3$ hooks to remove, so repeated application of the Murnaghan-Nakayama rule will yield a sum over all possible sequences of removing copies of the two skew $3$ hooks of length $2N-M$. For such a sequence $s\in\left\{\tiny{\Yvcentermath1\yng(1,1,1),\yng(3)}\right\}^{2N-M}$ we define the partition $\mu^s_p\models M$ to be the partition coming from $\lambda_p$ by consecutive removal of skew $3$ hooks as dictated by $s$. So we get:
\begin{equation*}
\chi^p(K_3) = \sum_s \chi^{\mu^s_p}(K_M)
\end{equation*}
where some care has to be taken: $\mu^s_p$ does not make sense for every sequence $s$, the numbers of copies of $\tiny{\Yvcentermath1\yng(1,1,1)}$ and $\tiny{\Yvcentermath1\yng(3)}$ that can be removed respectively are limited by functions of $p$. 

Next we need to know how often we obtain the same partition of $M$ in the sum above. First of all note that we only obtain tableaux of the form $(M-r,1^r)$ for some $r\geq 0$. Furthermore to obtain $(M-r,1^r)$ from $\lambda_p$ we certainly need that $p-r$ is positive and divisible by $3$. It is not difficult to see that we obtain $\binom{2N-M}{\frac{p-r}{3}}$ copies of each tableau that satisfies the conditions above. So we get:
\begin{equation*}
\chi^p(K_3) = \sum_{\substack{0\leq r \leq \min\{M-1,p\} \\ 3|p-r}} \chi^{r}(K_M) \binom{2N-M}{\frac{p-r}{3}}
\end{equation*}
which is the desired result.
\end{proof}

We write:
\begin{equation*}
s(N,W,m)=\sum\limits_{p=0}^{6N-2M-1}\frac{(-1)^p}{f^p} \chi^p (K_2)\chi^p(K_3)
\end{equation*}
So we have:
\begin{equation*}
s(N,W,m) = \sum\limits_{p=0}^{6N-2M-1} (-1)^{\floor{\frac{p}{2}}}   \sum_{\substack{0\leq r \leq \min\{M-1,p\} \\ 3|p-r}} \chi^{r}(K_M) \frac{\binom{2N-M}{\frac{p-r}{3}}\binom{3N-M-1}{\floor{\frac{p}{2}}}}{\binom{6N-2M-1}{p}}
\end{equation*}

\begin{lem} Let $W$ be a finite set of words in $L$ and $R$ not containing words of the form $[L^n]$ for any $n$, and $m\in\mathbb{N}^W$. Then:
$$
\lim\limits_{N\rightarrow\infty} s(N,W,m) = 2
$$
\end{lem}

\begin{proof}
First we look at the terms in the sum corresponding to $p=0$ and $p=6N-2M-1$. The sum of these two terms is equal to:
\begin{equation*}
\chi^0(K_M) + (-1)^{3N-M-1}\chi^{M-1}(K_M)
\end{equation*}
Recall from the proof of Lemma \ref{lem_alt} that $K_M$ is included in the alternating group if and only if $M$ is even. Also note that $\chi^{M-1}$ corresponds to the sign representation of $\symm{M}$. This means that:
\begin{equation*}
\chi^0(K_M) + (-1)^{3N-M-1}\chi^{M-1}(K_M) = 2
\end{equation*}
So what we need to prove is that the limit of the remaining terms is $0$. This follows easily from the fact that the $p^{th}$ term in the sum is of the order $\mathcal{O}(N^{-1})$ for $p=1$ and $p=6N-2M-2$, $\mathcal{O}(N^{-2})$ for $p=2$ and $p=6N-2M-3$ and smaller than the $p=2$ term for all $3\leq p \leq 6N-2M-4$
\end{proof}

Filling this in in the expression for $n(N,W,m)$ gives us that:
\begin{equation*}
n(N,W,m)  \sim 2 \frac{\aant{K_2}\cdot\aant{K(6N-2M)}}{(6N-2M)!}
\end{equation*}
for $N\rightarrow\infty$. General formulas for the cardinalities of conjugacy classes are known. These give:
\begin{equation*}
\aant{K_2}=\frac{(6N-2M)!}{2^{3N-M}(3N-M)!},\; \aant{K(6N-2M)} = (6N-2M-1)!
\end{equation*}
So we get:
\begin{equation*}
n(N,W,m) \sim \frac{(6N-2M)!!}{3N-M}
\end{equation*}
for $N\rightarrow\infty$, where for $t\in2\mathbb{N}$ the number $t!!$ is given by $(t-1)(t-3)\cdots 1$. 

This implies the following:

\begin{prp} Let $W$ be a finite set of words in $L$ and $R$ not containing words of the form $[L^n]$ for any $n$, $m\in\mathbb{N}^W$. Then:
$$
\lim\limits_{N\rightarrow\infty} \ExV{N}{\prod\limits_{w\in W}\left(Z_{[w]} \right)_{m_w}\; | \; g=\frac{N+1}{2}} = \lim\limits_{N\rightarrow\infty} \ExV{N}{\prod\limits_{w\in W}\left(Z_{[w]} \right)_{m_w}}
$$
\end{prp}

\begin{proof} We can again ignore representations of $(W,m)$ by intersecting circuits (one can show that these contribute a term of the order $\mathcal{O}(N^{-1})$), so we have:
\begin{align*}
\lim\limits_{N\rightarrow\infty} \ExV{N}{\prod\limits_{w\in W}\left(Z_{[w]} \right)_{m_w}\; | \; g=\frac{N+1}{2}} & = \lim\limits_{N\rightarrow\infty} a_{N,(W,m)} \cdot \frac{n(N,W,m)}{n(N,\emptyset,0)} \\
& = \lim\limits_{N\rightarrow\infty} a_{N,(W,m)}\frac{3N}{3N-M} \frac{(6N-2M)!!}{(6N)!!} \\
& = \lim\limits_{N\rightarrow\infty} a_{N,(W,m)} \frac{(6N-2M)!!}{(6N)!!}
\end{align*}
These are the same limits as in the unrestricted case (Theorem B from \cite{Pet}).
\end{proof}

Using the method of moments again, we immediately obtain the following:
\begin{thmB}
Let $W$ be a finite set of equivalence classes of words not containing $[L^n]$ for any $n$. Then we have: 
$$
\left(Z_{N,[w]}\right)\vert_{g=\frac{N+1}{2}}\rightarrow Z_{[w]} \text{ in distribution for }N\rightarrow\infty\text{ and }N\text{ odd}
$$
for all $[w]\in W$
\end{thmB}

\section{Corollaries}

In principle, Theorems A and B are topological theorems: they are about curves that `go in a certain direction' with respect to the orientation. However, they can be made geometrical by using properties of metrics on random surfaces coming from metrics on the triangle.

We can prove the following in the hyperbolic case:

\begin{cor1} Let $D_N\subset\mathbb{N}$ for all $N\in\mathbb{N}$ be a sequence of subsets such that one of the following holds:
\begin{itemize}
\item[1.] The sequence is non-negligible with respect to the genus.
\item[2.] $D_N=\left\{\frac{N+1}{2}\right\}$ for all odd $N$
\end{itemize}
Then both in the punctured and compactified hyperbolic setting we have that for all $\varepsilon>0$ sufficiently small and all $k\in\mathbb{N}$:
$$
\lim\limits_{N\rightarrow\infty} \ProC{N}{\abs{\sys-2\cosh^{-1}\left(\frac{k}{2}\right)}<\varepsilon}{g\in D_N} = \left(\prod\limits_{[w]\in\bigcup\limits_{i=3}^{k-1}A_i} \exp\left(-\frac{\aant{[w]}}{2\abs{w}}\right)\right) \left(1-\prod\limits_{[w]\in A_k}\exp\left(-\frac{\aant{[w]}}{2\abs{w}}\right)\right)
$$
where in the first case the limit has to be taken over even $N$ and over odd $N$ in the second case.
\end{cor1}
\begin{proof} It is not difficult to see that the expression on the right hand side is the asymptotic probability that the lowest trace that can be found on the surface is equal to $k$. In the punctured case this completes the prove because it is immediate that this curve corresponds to a systole.

The corresponding curve might however turn around some number of punctures in which case the curve is no longer essential when we project to the compactified surface. So in the compactified surface there remains something to prove. We treat the case of non-negligble $D_N$ and maximal genus separately.

First suppose that the sequence of sets $\rij{D_N}{N\in\mathbb{N}}{}$ is non-negligible. A curve that turns around some number of punctures must be separating. In this case it is either a left hand turn cycle on the dual graph or it is also separating on the dual graph. The first situation we have already excluded in the formula above. The second situation has asymptotic probability $0$ because of Theorem \ref{thm_sepcurves}, which is enough, because of our assumption on $\rij{D_N}{N\in\mathbb{N}}{}$. The $\varepsilon$-closeness of the systole follows from Proposition \ref{prp_cusplength}.

In the case of maximal genus we argue using the fact that there is only one left hand turn cycle. This immediately implies that any word we find in the graph corresponds to an essential curve. Furthermore, Lemma \ref{lem_Brooks} implies the $\varepsilon$-closeness of the systole.
\end{proof}

We also obtain a statement in the opposite direction to Corollary 1. That is, if we consider only surfaces with that satisfy certain conditions on the systole then the limits of the probabilities that these surface have a given genus do not change.

\begin{cor2} Let $D_N\subset\mathbb{N}$ for all $N\in\mathbb{N}$ be a sequence of subsets such that the probability $\Pro{N}{g\in D_N}$ converges for $N\rightarrow\infty$ and let $x\in (2\log((3+\sqrt{5})/2),\infty)$. Then both in the punctured and compactified hyperbolic setting we have: 
$$
\lim\limits_{N\rightarrow\infty} \ProC{N}{g\in D_N}{\sys \leq x} = \lim\limits_{N\rightarrow\infty} \Pro{N}{g\in D_N}
$$
and:
$$
\lim\limits_{N\rightarrow\infty} \ProC{N}{g\in D_N}{\sys \geq x} = \lim\limits_{N\rightarrow\infty} \Pro{N}{g\in D_N}
$$
\end{cor2}

\begin{proof} We first note that the conditions `$\sys\leq x$' and `$\sys\geq x$' can be expressed in terms of a finite number of $Z_{[w]}$-variables. We prove the corollary for $\sys\leq x$.

We first assume that $\lim\limits_{N\rightarrow\infty}\Pro{N}{g\in D_N} >0$. This means that:
\begin{align*}
\lim\limits_{N\rightarrow\infty} \ProC{N}{g\in D_N}{\sys \leq x} & = \lim\limits_{N\rightarrow\infty} \frac{\Pro{N}{g\in D_N \text{ and } \sys \leq x}}{\Pro{N}{\sys \leq x}} \\
& = \lim\limits_{N\rightarrow\infty} \frac{\Pro{N}{g\in D_N \text{ and } \sys \leq x}}{\Pro{N}{g\in D_N}\Pro{N}{\sys \leq x}}\Pro{N}{g\in D_N} \\
& = \lim\limits_{N\rightarrow\infty} \frac{\ProC{N}{\sys \leq x}{g\in D_N}}{\Pro{N}{\sys \leq x}}\Pro{N}{g\in D_N} \\
& = \lim\limits_{N\rightarrow\infty} \Pro{N}{g\in D_N} 
\end{align*}
where the last step follows from the fact that the sequence $D_N$ is non-negligble with respect to the genus and Corollary 1.

If $\lim\limits_{N\rightarrow\infty}\Pro{N}{g\in D_N} =0$ then we have:
\begin{align*}
\lim\limits_{N\rightarrow\infty} \ProC{N}{g\in D_N}{\sys \leq x} & \leq \lim\limits_{N\rightarrow\infty} \frac{\Pro{N}{g\in D_N}}{\Pro{N}{\sys \leq x}} \\
& = 0
\end{align*}
where we have used that $\lim\limits_{N\rightarrow\infty}\Pro{N}{\sys \leq x}>0$, which follows from the fact that $x\geq 2\log((3+\sqrt{5})/2)$.
\end{proof}

Note that in the proof of Corollary 2 we have only used the fact that our condition can be expressed in a finite number of $Z_{[w]}$-variables. This means that the corollary holds for all such conditions.

Another option is to endow our triangle with a more general Riemannian metric. This also induces a metric on every random surface. Although we lose some of the nice combinatorial properties of lengths of curves, we can still prove the following:

\begin{cor3}Let $D_N\subset\mathbb{N}$ for all $N\in\mathbb{N}$ be a sequence of subsets such that one of the following holds:
\begin{itemize}
\item[1.] The sequence is non-negligible with respect to the genus.
\item[2.] $D_N=\left\{\frac{N+1}{2}\right\}$ for all odd $N$
\end{itemize}
Then:
$$
\lim\limits_{N\rightarrow\infty} \ProC{N}{\sys < m_1(d)}{g\in D_N} = 0
$$ 
and for all $x\in [0,\infty)$:
$$
\lim\limits_{N\rightarrow\infty} \ProC{N}{\sys \geq x}{g\in D_N} \leq 1-\sum\limits_{k=2}^{\floor{x/m_2(d)}}  \left(e^{-\sum\limits_{j=1}^{k-1}\frac{2^{j-1}-1}{j}} -e^{-\sum\limits_{j=1}^k\frac{2^{j-1}-1}{j}}\right)
$$
where both limits have to be taken over even $N$ in the first case and over odd $N$ in the second case
\end{cor3}

\begin{proof} The first inequality follows from the fact that the systole needs to cross at least two triangles and cannot turn around a vertex. Namely, this implies that for any $\omega\in\Omega_N$ we have:
\begin{equation*}
\sys(\omega)\geq m_1(d)
\end{equation*}

To prove the second inequality we note that if the graph dual to the triangulation of a random surface contains a circuit of $k$ edges that cannot be contracted to a point on the surface we have:
\begin{equation*}
\sys(\omega) \leq m_2(d) \cdot k
\end{equation*}
So we get:
\begin{equation*}
\ProC{N}{\sys \geq x}{g\in D_N} \leq \ProC{N}{\Gamma \text{ contains no essential circuit of }\leq \floor{\frac{x}{m_2(d)}}\text{ edges} }{g\in D_N}
\end{equation*}
If we assume for a moment that all the circuits of length up to $\floor{\frac{x}{m_2(d)}}$ edges except those carrying words of the for $[L^n]$ for some $n$ are essential then the probability on the right hand side above has been worked out in Section 6.1 of \cite{Pet}. To prove that we can make this assumption we can use exactly the same arguments as in the proof of the previous corollary.
\end{proof}

We note that formulas similar to those in Corollaries 1 and 3 can be worked out for the probability distribution of the $n^{th}$ shortest curve up to any finite $n$ using exactly the same techniques.

As a final remark we note that the upper bound in Corollary 3 is sharp: using the construction from Section 6.2 of \cite{Pet} we can approach the upper bound arbitrary closely.

\nocite{*}

\end{document}